\documentclass[12pt,reqno]{amsart}

\usepackage{amsmath}
\usepackage{amsfonts}
\usepackage{amssymb}
\usepackage{amscd}
\usepackage{hyperref,color}

\newcommand{\func}{\operatorname}
\renewcommand{\bot}{\perp}
\newtheorem{theorem*}{}
\newtheorem{theorem}{Theorem}
\newtheorem{corollary}[theorem]{Corollary}
\newtheorem{lemma}[theorem]{Lemma}
\newtheorem{proposition}[theorem]{Proposition}

\newtheorem{definition}[theorem]{Definition}
\newtheorem{remark}[theorem]{Remark}
\newtheorem{notation}[theorem]{Notation}

\numberwithin{equation}{section}
\newcommand{\limfunc}{\operatorname}

\begin{document}
\title{Laplace and Dolbeault operators on Sasakian manifolds}

\author[G.~Habib]{Georges Habib}
\address{Lebanese University \\
Faculty of Sciences II \\
Department of Mathematics\\
P.O. Box 90656 Fanar-Matn \\
Lebanon\\ 
And Universit\'e de Lorraine\\
CNRS, IECL\\
54506 Nancy, France}
\email[G.~Habib]{ghabib@ul.edu.lb}
\author[K.~Richardson]{Ken Richardson}
\address{Department of Mathematics \\
Texas Christian University \\
Fort Worth, Texas 76129, USA}
\email[K.~Richardson]{k.richardson@tcu.edu}
\author[R.~Wolak]{Robert Wolak}
\address{Wydzial Matematyki i Informatyki UJ \\
Ul. prof. Stanislawa Lojasiewicza 6 \\
30-348 Krakow, Poland}
\email[R.~Wolak]{robert.wolak@uj.edu.pl}

\begin{abstract}
    Sasakian manifolds, the odd-dimensional analogues 
    of K\"ahler manifolds, carry two natural pairs of 
    first-order Dolbeault-type operators on the full 
    complex of differential forms, extending 
    the Kohn--Rossi differentials. On 
    such Sasaki manifolds, we establish K\"ahler-type 
    identities for these operators, and relate the 
    resulting Dolbeault Laplacians to the Hodge 
    Laplacian. Unlike the K\"ahler case, where 
    $\Delta=2\Delta_{\overline\partial}$, the relation 
    has extra terms coming from the Reeb flow. 
    Using these formulas and a Lefschetz decomposition,
    we derive lower and upper eigenvalue 
    estimates for $\Delta$ on forms depending on the 
    eigenvalues of the Lie derivative of the Reeb 
    vector field.
\end{abstract}



\maketitle
\tableofcontents

\section{Introduction}
Sasakian manifolds, introduced by Sasaki \cite{Sasaki1960OnDiffble} in 1960, can be understood as odd dimensional counterparts  of K\"{a}hler manifolds.  In  the 1990s, new   physical conjectures  brought a renewed interest in Sasakian geometry. Of particular interest is  the Maldacena
conjecture, cf.  \cite{Maldacena1998LargeNSCFT},  on the duality between conformal field theory and supergravity on anti-de-Sitter space-time. The Maldacena conjecture, also called AdS/CFT correspondence, proposes that a universe with gravity in higher dimensions is mathematically equivalent to a flat quantum universe with no gravity in lower dimensions. Thus it provides a relation between the string theory and the quantum field theory. For the review of recent developments in the theory see \cite{gauntlett2025sasakieinsteingeometrygkgeometry}, and some new results  e.g. \cite{MartelliSparksYau2008}, \cite{GenoliniPerezSparks2019}.





Let $(M^{2n+1},g,\xi )$ be a Sasakian manifold. 
This means that the vector field $\xi$ is a unit Killing vector field such that
the endomorphism $\phi :TM\rightarrow TM$ defined by $\phi \left( \bullet \right) :=\nabla _{\bullet }\xi$ satisfies, for $X,Y\in \Gamma \left( TM\right) $, the relations
\begin{eqnarray*}
\phi ^{2} &=&-I+\eta \otimes \xi  \\
\left( \nabla _{X}\phi \right) \left( Y\right) &=&\eta \left( Y\right)
X-g\left( X,Y\right) \xi 
\end{eqnarray*}%
where $\eta=\xi^\flat$ is the form associated to $\xi$ by the musical isomorphism. Here $\nabla$ is the Levi-Civita connection on $M$. The second identity, known as the normality condition, can be equivalently expressed as $\nabla_X d\eta=2\eta\wedge X^\flat$. An easy consequence of the definition is that the following identity
$$g\left( \phi X,\phi Y\right) =g\left( X,Y\right) -\eta \left( X\right)\eta \left( Y\right),$$
holds, for all vector fields $X, Y$. 

\begin{remark}
We are using notation consistent with that used in, for example, \cite%
{Fujitani1966ComplexValuedDiffFormsSasaki} and \cite%
{Cappelletti-MontanoEtAl2015HardLefschetzSasakian}. Other researchers use
metrics that differ from ours by a constant, e.g.  \cite%
{Blair2010RiemGeomContactSymplMflds}. We also note that some authors choose
the opposite sign for the transverse complex structure $\phi \left( \bullet
\right) =-\nabla _{\bullet }\xi $, corresponding to the normality condition $%
\left( \nabla _{X}\phi \right) \left( Y\right) =g\left( X,Y\right) \xi -\eta
\left( Y\right) X$. We made our choice so that the resulting K\"{a}hler identities have the
same sign as on K\"{a}hler manifolds.
\end{remark}

The foliation by the integral curves of $\xi$ gives rise to a transverse K\"{a}hler foliation whose complex structure is the endomorphism $\phi$. In particular, it implies that any embedded manifold transverse to the foliation inherits the canonical K\"{a}hler structure. An equivalent definition for a Sasakian manifold is that the Riemannian cone $(C(M)=(0,\infty)\times M, dt^2+t^2g)$ is a K\"ahler manifold. 
Over the years Sasakian manifolds have been studied as

\begin{enumerate}
\item[(I)] foliated Riemannian manifolds with a very particular 1-dimensional Riemannian foliation (isometric flow) with a transverse K\"{a}hler structure;

\item[(II)] a one-dimensional CR structure; 

\item[(III)]
a special codimension 1 submanifold of a K\"{a}hler manifold.
\end{enumerate}

The foliated approach concentrated on the investigation of the foliated (transverse) objects and their properties like basic cohomology, transverse Hodge theory, basic Dolbeault cohomology, basic Laplacian (acting on basic forms) and its spectrum. A lot of effort has been put into finding necessary and sufficient cohomological conditions for a foliated Riemannian manifold to be Sasakian. The CR approach is more general; in this setting, the emphasis is on horizontal forms and operators acting on them. In the cone approach, we consider operators which are restrictions of well-known operators in K\"{a}hler geometry to differential forms on  the Sasakian manifold seen as a codimension one submanifold.

In this paper, we provide a unified approach to various Dolbeault-type cohomologies studied in multiple previous works. We explain the origins  of the assumed definitions, prove numerous  useful formulas, both well-known or slightly more general as well as totally new ones. On the way we correct some mistakes and show that some of these formulas are true not only for Sasaki-Einstein manifolds but also for  general Sasakian manifolds. After establishing known and new identities, we develop a theory of Dolbeault type operators and associated cohomologies on Sasakian manifolds.


In Section~\ref{differentials section}, we introduce the fundamental objects on Sasakian manifolds. The complex  cotangent bundle of a Sasakian manifold of dimension $2n+1$ splits into 3 components of  dimensions $n$, $n$, and $1$,  respectively. This is derived from the splitting of the tangent bundle of the manifold  into the tangent and the normal bundle of the foliation, and then of the complex  normal bundle into the eigenbundles of  its complex structure. Thus, as the foliation is one-dimensional, the space $\Omega^m$ of complex-valued differential $m$-forms splits into a direct sum 
$$\bigoplus_{r+s=m} \Omega^{r,s,0} \oplus \bigoplus_{r+s=m-1} \Omega^{r,s,1}.$$
With this gradation, the differential $d$ has 4 components and splits as

$$d= d^{1,0,0}+d^{0,1,0}+d^{0,0,1}+d^{1,1,-1}.$$

 This is mainly the subject of Section~\ref{differentials section}, where we consider two different pairs 
 of Dolbeault-type operators acting on all forms, namely $ \partial    , \bar{ \partial  }$ and $ \partial    _{C}, \bar{ \partial    }_{C}$. These are extensions of the Kohn-Rossi differentials considered in \cite{KohnRossi1965ExtHolomFcnsToBndryCpxMfd}, which by definition act only on horizontal forms, i.e. on $\bigoplus_{r+s=k} \Omega^{r,s,0}$. The operators 
\begin{equation*}
\partial =d^{1,0,0},\qquad \overline{\partial }=d^{0,1,0}
\end{equation*}
are natural first-order differentials on the space of all forms, in analogy to the Dolbeault differentials on complex manifolds.

The basic forms of degree $m$ corresponding to the Riemannian foliation lie in $\bigoplus_{r+s=m} \Omega^{r,s,0}$, and the differential has just two components $d= \partial_b+\overline\partial_b$. Therefore the Dolbeault theory for basic forms of the characteristic foliation of a Sasakian manifold is not much different from that of a K\"{a}hler manifold. In fact,  this is the case for a more general situation: transversely K\"{a}hler, taut Riemannian foliations of a compact manifold.

The other pair of differentials has been investigated by J. Schmude, e.g. \cite{Schmude2014LaplaceOpsSasakiEinstMflds}, where  the author uses the Dolbeault operators which we denote by $\partial_C$ and $\overline{\partial_C}$, which turn out to be the restrictions of the Dolbeault operators on the cone $C(M)$ with the induced K\"ahler structure; see Section~\ref{appendix} for a detailed discussion. These operators can be represented as:

\begin{eqnarray*}
\partial _{C} &=&d^{1,0,0}+\frac{1}{2}d^{1,1,-1}=\partial + \frac 12(d\eta\wedge)(\xi\lrcorner),
\\
\overline{\partial _{C}}&=&d^{0,1,0}+\frac{1}{2}d^{1,1,-1}
=\overline\partial+ \frac 12(d\eta\wedge)(\xi\lrcorner).
\end{eqnarray*}%

 We formulate and prove numerous useful formulas involving the structure tensors of the Sasakian manifold. 
 Proposition \ref{Lefschetz proposition}  presents the Lefschetz decomposition for differential forms, Proposition~\ref{Kaehler Identity prop} the  K\"{a}hler  identities and  Proposition~\ref{del bar C commutator formulas}  general versions of identities proved earlier  for Sasaki-Einstein manifolds.

 Section~\ref{Sasaki section} is dedicated to the use of these identities in order to investigate the relation between the Dolbeault-Laplace operators $\Delta_\partial, \Delta_{\overline{\partial}}, \Delta_{\partial_C}, \Delta_{\overline{\partial}_C}$ defined by the differentials $ \partial    $, $\bar{ \partial    }$,  ${ \partial    }_C$, $\bar{ \partial    }_C$, and  the  Hodge Laplacian operator 
 $\Delta$ of the Sasakian manifold. In contrast to the situation of  a K\"ahler manifold, where $\Delta=2\Delta_\partial=2\Delta_{\overline{\partial}}$, the formulas are more complicated on Sasaki manifolds. In particular, we show the following:

\begin{theorem} \label{Laplacian theorem}
Let $(M^{2n+1},g,\xi )$ be a Sasakian manifold. The Dolbeault Laplacians $\Delta
_{\partial }$ and $\Delta _{\overline{\partial }}$ are related by 
\begin{equation}
\Delta _{\partial }=\Delta _{\overline{\partial }}-2iH\mathcal{L}_{\xi }.
\label{laplaceDels}
\end{equation}%
The Hodge Laplace operator acting on differential forms on $M$ satisfies 
\begin{equation}
\Delta =2\Delta _{\overline{\partial }}-2iH\mathcal{L}_{\xi }+2i\eta \wedge
\left( \overline{\partial }^{\ast }-\partial ^{\ast }\right) +2i\xi
\lrcorner \left( \overline{\partial }-\partial \right) +4L\Lambda +4\eta
\wedge \xi \lrcorner H-\mathcal{L}_{\xi }^{2}.  \label{eq:laplaceDel}
\end{equation}
\end{theorem}


Here, $H, \Lambda$ and $L$ are operators acting on differential forms defined in Section \ref{differentials section} and $\mathcal{L}_{\xi }$ is the Lie derivative in the direction of $\xi$. In the following section we develop the de Rham-Hodge and spectral theory for the Laplace operator associated with $ \partial    , \bar{ \partial    }$ and $ \partial    _{C}, \bar{ \partial    }_{C}$. We also use the fact that the flow of the characteristic vector field $\xi$ defines an abelian subgroup of the group of isometries of the Sasakian manifold. Its closure is a torus contained in $\mathrm{Isom}(M,g)$. Theorem~\ref{rst Hodge Theorem} summarizes
the consideration of Section \ref{Spectral theory section}.

\begin{theorem}\label{rst Hodge Theorem} Let $(M^{2n+1},g,\xi )$ be a compact Sasakian
manifold. For $0\leq r,s\leq n$ and $t\in \{0,1\}$, let $H_{\overline{%
\partial }}^{r,s,t}\left( M\right) $, respectively $H_{\partial
}^{r,s,t}\left( M\right) $, be the cohomology groups defined by%
\begin{equation*}
H_{\overline{\partial }}^{r,s,t}\left(M\right) =\frac{\ker(\overline{%
\partial }:\Omega ^{r,s,t}\rightarrow \Omega ^{r,s+1,t})}{\limfunc{im}%
(\overline{\partial }:\Omega ^{r,s-1,t}\rightarrow \Omega ^{r,s,t})},~\text{%
respectively }H_{\partial }^{r,s,t}\left( M\right) =\frac{\ker (\partial
:\Omega ^{r,s,t}\rightarrow \Omega ^{r+1,s,t})}{\limfunc{im}(\partial :\Omega
^{r-1,s,t}\rightarrow \Omega ^{r,s,t})}.
\end{equation*}%
Then the natural map $\Omega^{r,s,t}\to  H_{\bullet
}^{r,s,t}\left( M\right); \omega \mapsto \left[ \omega \right] $ induces isomorphisms 
\begin{eqnarray*}
\mathcal{H}_{\overline{\partial }}^{r,s,t} &=&\ker (\Delta _{\overline{%
\partial }}^{r,s,t})=\ker (\left. \overline{\partial }\right\vert _{\Omega
^{r,s,t}})\cap \ker (\left. \overline{\partial }^{\ast }\right\vert _{\Omega
^{r,s,t}})\cong H_{\overline{\partial }}^{r,s,t}\left( M\right) , \\
\mathcal{H}_{\partial }^{r,s,t} &=&\ker (\Delta _{\partial }^{r,s,t})=\ker
(\left. \partial \right\vert _{\Omega ^{r,s,t}})\cap \ker (\left. \partial
^{\ast }\right\vert _{\Omega ^{r,s,t}})\cong H_{\partial }^{r,s,t}\left(
M\right) ,
\end{eqnarray*}%
which may be infinite dimensional. There is a countable set of charges $q\in 
\mathbb{R}$ such that the $iq$-eigenspace $\Omega ^{r,s,t}(q)$ of $\mathcal{L%
}_{\xi }$ on $\Omega ^{r,s,t}$ is nonzero, and the spaces above are Hilbert
sums of the spaces%
\begin{eqnarray*}
\mathcal{H}_{\overline{\partial }}^{r,s,t}(q) &=&\ker (\left. \Delta _{%
\overline{\partial }}\right\vert _{\Omega ^{r,s,t}(q)})\cong H_{\overline{%
\partial }}^{r,s,t}\left( q\right) , \\
\mathcal{H}_{\partial }^{r,s,t}(q) &=&\ker (\left. \Delta _{\partial
}\right\vert _{\Omega ^{r,s,t}(q)})\cong H_{\overline{\partial }%
}^{r,s,t}\left( q\right) ,
\end{eqnarray*}%
all of which are finite dimensional.
\end{theorem}

 Section~\ref{Kohn Rossi section} explains the relation between the cohomology defined by the $\partial$-operator investigated in this paper and the Kohn-Rossi cohomology as well as with the basic Dolbeault cohomology.  In Section~\ref{Lefschetz conseq section}, we use the Lefschetz operator $L$
 and the  decomposition theorem for differential forms  to represent $\bar{ \partial    }$-harmonic forms in terms of primitive forms, meaning differential forms belonging to the kernel of the operator $\Lambda$. Unlike the K\"ahler manifold situation where $L$ commutes with $\Delta_{\overline\partial}$,  the operator $L$ shifts the $\Delta_{\overline\partial}$-eigenvalues on a Sasakian manifold.
\begin{theorem}\label{Lefschetz theorem}
Let $(M^{2n+1},g,\xi )$ be a compact Sasakian
manifold. Let $m\in\{1,\ldots, 2n+1\}$. For any $q\in\mathbb{R}$, we have the decomposition
\begin{equation*}
\mathcal{H}_{\overline{\partial }}^{m}\left( q\right) =\mathcal{H}_{%
\overline{\partial },P}^{m}\left( q\right) \oplus
\bigoplus\nolimits_{1\leq b\leq \left\lfloor \frac{m}{2} \right\rfloor
}L^{r}E_{-2bq}\left( \Delta _{\overline{\partial },P}\left(
q\right) \right) ,
\end{equation*}%
where $E_{\lambda }\left( \Delta _{\overline{\partial },P}(q)\right)$ denotes the eigenspace of $\Delta_{\overline\partial}$ restricted to primitive forms corresponding to eigenvalue $%
\lambda $. In particular, for $q>0$, $\mathcal{H}_{\overline{\partial }%
}^{m}\left( q\right) =\mathcal{H}_{\overline{\partial },P}^{m}\left(
q\right) $. 
\end{theorem}

The main results of Section~\ref{eigenvalue estimates section} are to provide upper bounds for the first eigenvalue of the Hodge Laplace operator based on a careful use of Equation \eqref{eq:laplaceDel} and the min-max principle for a particular test-form. We first state a lower bound for the eigenvalues and characterize its limiting case.

\begin{theorem}\label{thm:eigenestimatedeltabar}
Let $(M^{2n+1},g,\xi )$ be a compact Sasakian
manifold. Let $m\in\{0,\ldots, n-1\}$. If  $\omega$ is a nonzero eigenform of $\Delta_{\overline\partial}$ restricted to $m$-forms, that is $\Delta_{\overline\partial}\omega=\lambda\omega$ with $\mathcal{L}_\xi\omega=iq\omega$, for some $q\leq 0$, then the following estimate
$$\lambda\geq -q(n-m).$$
holds. Equality is attained if and only if $\lambda=q=0$, that is, $\omega\in \mathcal{H}^m_{\overline\partial}(0)$.
\end{theorem}
For upper bounds of the first positive eigenvalue $\lambda_{1,k}$ of the Hodge Laplacian restricted to $k$-forms, we get
\begin{theorem} \label{eigenvalue estimate 1}
Let $(M^{2n+1},g,\xi )$ be a compact Sasakian manifold. Let $m\in \{0,\ldots, 2n+1\}$. Then 
\begin{itemize} 
\item If $m>n+1$ we have $\mathcal{H}_{\overline\partial,P}^{m}(0)=0$,   and $\mathcal{H}_{\overline\partial,P}^{n+1}(0)\subset H^{n+1}(M)$. 
\item If $\mathcal{H}_{\overline\partial,P}^{m}(0)$ is nonzero
for some $m<n$, the first positive eigenvalue $\lambda_{1,m+1}$  satisfies
$$\lambda_{1,m+1}<4(n-m+1).$$
Also, we have 
$$\lambda_{1,m+2b}< 4(b+1)(n-m-b+1),$$
for any $0<b\leq\frac{n-m}{2}$. 
\end{itemize}
\end{theorem} 
When the charge is not zero, we get the following estimate for the eigenvalues
\begin{theorem} \label{eigenvalue estimate 2}Let $(M^{2n+1},g,\xi)$ be a compact Sasakian manifold. Let  $m\in \{0,\ldots, 2n+1\}$. Then, 
\begin{itemize} 
\item If $m<n$ and $q<0$ (resp. $m>n+1$ and $q>0)$, we have $\mathcal{H}%
_{\overline{\partial}}^{m}(q)=0$. 
\item If $m=n$, $q<0$ and $\mathcal{H}_{\overline{\partial}%
}^{m}(q)\neq 0$ (resp. $m=n+1$ and $q>0$), then $q^2$ is an eigenvalue of the Hodge Laplacian restricted to $m$-differential forms. 
\item When $\mathcal{H}_{\overline\partial}^{m}(q)$ is nonzero for some $q>0$ and $m<n+1$,  then the first positive eigenvalue $\lambda_{1,m}$  satisfies
$$\lambda_{1,m}<q^2+2q(n-m+1),$$
and, for $1\leq b\leq n-m$, 
$${\rm min}(\lambda_{1,m+2b}, \lambda_{1,m+2b+1})<q^2+2q(n-m+1)+4(b+1)(n-m-b+1).$$
\end{itemize}
\end{theorem}

Whereas on a K\"ahler manifold the operators $\Delta$ and $2\Delta_{\overline\partial}$ are the same, the following theorem shows how to obtain some eigenvalues of $\Delta$ on forms from eigenvalues of $\Delta_{\overline\partial}$ on functions on a compact Sasakian manifold.
\begin{theorem} \label{eigenvalue}
Let $(M^{2n+1},g,\xi)$ be a compact Sasakian manifold. Let $f$ be an eigenfunction of the ${\overline\partial}$-Laplacian associated with an eigenvalue $q$ of the Lie derivative, that is, $\Delta_{\overline\partial} f=\lambda f$ and  $\xi(f)=iqf$. Let $m\in\{1,\ldots,n\}$. 
Then there exist four linear combinations of the forms
$\alpha_m:=L^mf$, $\eta\wedge \overline\partial^*\alpha_m$, 
$\eta\wedge \partial^*\alpha_m$, and 
$\overline\partial\partial\alpha_{m-1}$ that are
eigenforms of degree $2m$ of the Hodge Laplacian corresponding 
to eigenvalues
$$\Theta_0(q,\lambda)=2\lambda-4m^2+4mn+2nq+q^2,$$
and/or 
$$\Theta_\pm(q,\lambda)=2\lambda-4m^2+4mn+4m+2nq-2n+q^2\pm 2\sqrt{2\lambda+(n+q)^2}.$$
When all of the forms $\overline\partial^*\alpha_m$, 
$\eta\wedge \partial^*\alpha_m$, and 
$\overline\partial\partial\alpha_{m-1}$ are nonzero, 
$\Theta_0(q,\lambda)$ has multiplicity two, and 
each $\Theta_\pm(q,\lambda)$ has multiplicity one.
\end{theorem}

The final section is dedicated to the presentation of the investigated theory for the particular Sasakian manifold: the round 3-sphere equipped with the standard metric and Reeb vector field defining the Hopf fibration. 
In this example, we show that the cohomologies associated with the $\overline\partial$ are not zero and discuss the applications of Theorem \ref{eigenvalue estimate 2}. 
Finally, for the reader's convenience we have included an appendix presenting the detailed description of the Dolbeault operators on the K\"{a}hler cone which elucidates the formulae  for $ \partial_C$ and $\overline{ \partial}_C$ operators  defined on a Sasakian manifold.

{\bf Acknowledgment:} This work received funding by the CY Initiative of Excellence (grant ``IDEE@CY'' ANR-20-IDES-0004) and was developed during the stay of the first two authors at CY Advanced Studies whose support is gratefully acknowledged. The authors acknowledge that the research cooperation was funded by the program Excellence Initiative
– Research University at the Jagiellonian University in Krakow within the framework of the research group
Reeb-Reinhart 2022. The first named author was also supported by the french CNRS through the grant ``Postes Rouges''  during his stay at University of Lorraine. 

\section{Differentials and K\"ahler identities on Sasakian manifolds}\label{differentials section}

Let $(M^{2n+1},g,\xi )$ be a compact Sasakian manifold. We denote by $\eta
=\xi ^{\flat }$ the $1$-form corresponding to the Killing vector field $\xi $
by the musical isomorphism and by $\phi =\nabla \xi $, the transverse
complex structure. Since we have that $\phi ^{2}=-\mathbf{1}$ when
restricted to $D=\left( \ker \eta \right) \otimes \mathbb{C}$ and vanishes
on $\func{span}\nolimits_{\mathbb{C}}\left\{ \xi \right\} $, we let%
\begin{eqnarray*}
D^{1,0} &=&+i~\text{eigenspace of }\phi ; \\
D^{0,1} &=&-i~\text{eigenspace of }\phi ,
\end{eqnarray*}%
so that%
\begin{equation*}
TM\otimes \mathbb{C}=\func{span}\nolimits_{\mathbb{C}}\left\{ \xi \right\}
\oplus D^{1,0}\oplus D^{0,1},
\end{equation*}%
and, thus, 
\begin{eqnarray*}
T^{\ast }M\otimes \mathbb{C} &=&\func{span}_{\mathbb{C}}\{\eta \}\oplus
D^{\ast } \\
&=&\func{span}_{\mathbb{C}}\{\eta \}\oplus D^{\ast 1,0}\oplus D^{\ast 0,1},
\end{eqnarray*}%
with%
\begin{eqnarray*}
D^{\ast 1,0} &=&\left\{ \alpha \in D^{\ast }:\alpha \left( Z\right) =0\text{~%
}\forall Z\in D^{0,1}\right\} , \\
D^{\ast 0,1} &=&\left\{ \alpha \in D^{\ast }:\alpha \left( Z\right) =0\text{~%
}\forall Z\in D^{1,0}\right\} .
\end{eqnarray*}%
The inner product on $T^{\ast }M\otimes \mathbb{C}$ induces an inner product
on the space of all differential forms. We write%
\begin{eqnarray*}
\Omega ^{r,s,0} &=&\Gamma \left( \Lambda ^{r}D^{\ast 1,0}\wedge \Lambda
^{s}D^{\ast 0,1}\right) , \\
\Omega ^{r,s,1} &=&\Gamma \left( \eta \wedge \Lambda ^{r}D^{\ast 1,0}\wedge
\Lambda ^{s}D^{\ast 0,1}\right) .
\end{eqnarray*}%
We see that the space $\Omega ^{m}$ of complex-valued $m$-forms can be written as%
\begin{equation*}
\Omega ^{m}=\bigoplus_{r+s=m}\Omega ^{r,s,0}\oplus \bigoplus_{r+s=m-1}\Omega
^{r,s,1}=\bigoplus_{\substack{ r+s+t=m \\ 0\leq t\leq 1}}\Omega ^{r,s,t}.
\end{equation*}%

Let $\Pi ^{r,s,t}:\Omega ^{\ast }\rightarrow \Omega ^{r,s,t}$ be the
orthogonal projection. According to the above splitting, we  decompose $d$ as 
\begin{equation*}
d=\sum_{a,b,c}d^{a,b,c},
\end{equation*}%
where 
\begin{eqnarray*}
d &:&\Omega ^{r,s,t}\rightarrow \bigoplus_{\substack{ a+b+c=1 \\ -1\leq
a,b,c\leq 2}}\Omega ^{r+a,s+b,t+c},\text{ or} \\
\left. d\right\vert _{\Omega ^{r,s,t}} &=&\sum_{_{\substack{ a+b+c=1 \\ %
-1\leq a,b,c\leq 2}}}\Pi ^{r+a,s+b,t+c}d=\sum_{_{\substack{ a+b+c=1 \\ %
-1\leq a,b,c\leq 2}}}\Pi ^{+a,+b,+c}d,
\end{eqnarray*}%
with 
\begin{equation*}
\left. d^{a,b,c}\right\vert _{\Omega ^{r,s,t}}=\Pi ^{+a,+b,+c}d.
\end{equation*}%
By the given construction, the only nonzero terms among the $d^{a,b,c}$ are $%
d^{1,0,0}$,$~d^{0,1,0}$, $d^{0,0,1}$, $d^{1,1,-1}$, because $d\eta $ is of
type $\left( 1,1,0\right) $. Hence, we write 
\begin{equation*}
d=d^{1,0,0}+d^{0,1,0}+d^{0,0,1}+d^{1,1,-1}.
\end{equation*}

As $d^{2}=0$, we get after comparing the different degrees that 
\begin{equation*}
\left( d^{1,0,0}\right) ^{2}=\left( d^{0,1,0}\right) ^{2}=\left(
d^{0,0,1}\right) ^{2}=\left( d^{1,1,-1}\right) ^{2}=0
\end{equation*}%
as well as 
\begin{equation*}
\left\{ d^{0,1,0},d^{1,0,0}\right\} +\left\{ d^{1,1,-1},d^{0,0,1}\right\}
=\left\{ d^{1,0,0},d^{1,1,-1}\right\} =\left\{ d^{0,1,0},d^{0,0,1}\right\}
=\left\{ d^{0,1,0},d^{1,1,-1}\right\} =0.
\end{equation*}%
Thus, the operators 
\begin{equation*}
\partial =d^{1,0,0},\qquad \overline{\partial }=d^{0,1,0}
\end{equation*}%
are natural first-order differentials acting on the space of all forms, in analogy
to the Dolbeault differentials on complex manifolds, and the other
differentials satisfy 
\begin{eqnarray*}
d^{0,0,1} &=&\eta \wedge \mathcal{L}_{\xi }, \\
d^{1,1,-1} &=&d\eta \wedge (\xi \lrcorner )=2L(\xi \lrcorner ),
\end{eqnarray*}%
where $\mathcal{L}_{\xi }$ denotes the Lie derivative, and $L:=\frac{1}{2}%
\left( d\eta \wedge \right) $ on the set of all differential forms. We use the
symbol $\lrcorner $ to denote interior product with either a vector field or a
differential form, where we have chosen the sign so that for any vector
field $Y$ and differential form $\omega $ 
\begin{equation*}
Y\lrcorner \omega =Y^{\flat }\lrcorner \omega ,
\end{equation*}%
and $\alpha \lrcorner =(\alpha \wedge )^{\ast }$ for a differential form $\alpha 
$.  From the above, we also have 
\begin{equation}\label{eq:partialoverpartial}
\overline{\partial }\partial +\partial \overline{\partial }=-\left\{
d^{1,1,-1},d^{0,0,1}\right\} =-2L\mathcal{L}_{\xi }.
\end{equation}

Other differentials have been investigated by \cite%
{Schmude2014LaplaceOpsSasakiEinstMflds} and others, where instead the author
uses the Dolbeault operators that we denote by $\partial_C$ and $\overline{%
\partial_C}$, which turn out to be the restriction of the Dolbeault
operators on the cone $C(M)=(0,\infty)\times M$ with the induced K\"ahler
structure; see Section~\ref{appendix} for a detailed discussion. These
operators are also given as a combination of the forementioned operators 
\begin{eqnarray*}
\partial _{C} &=&d^{1,0,0}+\frac{1}{2}d^{1,1,-1}=\partial + L(\xi\lrcorner),
\\
\overline{\partial _{C}}&=&d^{0,1,0}+\frac{1}{2}d^{1,1,-1}
=\overline\partial+L(\xi\lrcorner).
\end{eqnarray*}%
A direct consequence of the above formulas is that $\partial_C^2=\overline
\partial _{C}^{2}=0$. For any smooth function $f$, we have 
\begin{equation*}
df=\xi \left( f\right) \eta +\partial _{C}f+\overline{\partial _{C}}f.
\end{equation*}
We note that $d$, $\partial$, $\overline{\partial}$, $d^{0,0,1}$, $\partial
_{C}$, $\overline{\partial _{C}}$ all obey the Leibniz rule. We emphasize
that the operators $\partial$ and $\partial_C$ (and thus $\overline{\partial}
$ and $\overline{\partial_C}$) are two different extensions to all
differential forms of the Kohn-Rossi differentials in \cite%
{KohnRossi1965ExtHolomFcnsToBndryCpxMfd}, which are applied to horizontal
forms only and have been studied by many authors in the setup of CR manifolds.

We let $\partial ^{\ast }$ and $\overline{\partial }^{\ast }$ denote the
adjoints of $\partial $ and $\overline{\partial }$ with respect to the inner
product on all forms. Also, let 
\begin{equation*}
\Lambda:=L^*.
\end{equation*}
We now compute the various commutators of the operators we have introduced.
First, we see that 
\begin{equation}
\lbrack \mathcal{L}_{\xi },\tau ]=0  \label{Lie deriv commutes with all}
\end{equation}%
for the operators $\tau =d,\partial ,\overline{\partial },L,\Lambda ,\eta
\wedge ,\xi \lrcorner $, and 
\begin{equation*}
\mathcal{L}_{\xi }^{\ast }=-\mathcal{L}_{\xi }.
\end{equation*}%
Since $\partial \eta =0$, we can deduce that 
\begin{gather}
\left\{ \partial ,\xi \lrcorner \right\} =\left\{ \overline{\partial },\xi
\lrcorner \right\} =0,\left\{ \partial ^{\ast },\eta \wedge \right\}
=\left\{ \overline{\partial }^{\ast },\eta \wedge \right\} =0  \notag \\
\left\{ \partial ,\eta \wedge \right\} =\left\{ \overline{\partial },\eta
\wedge \right\} =0,\,\,\left\{ \partial ^{\ast },\xi \lrcorner \right\}
=\left\{ \overline{\partial }^{\ast },\xi \lrcorner \right\} =0,
\label{del xi eta commutators}
\end{gather}%
where in each line, the last two equations are obtained from the first two
by taking the adjoint. We have the following commutation relations:%
\begin{equation*}
0=\left[ L,\partial \right] =\left[ L,\overline{\partial }\right] =\left[
\Lambda ,\overline{\partial }^{\ast }\right] =\left[ \Lambda ,\partial
^{\ast }\right] .
\end{equation*}%
We have used the fact that $L$ commutes with $d$ (since $d\eta $ is a closed 
$\left( 1,1,0\right) $-form) and thus with each of $d^{1,0,0}$, $d^{0,1,0}$, 
$d^{0,0,1}$, $d^{1,1,-1}$.

\begin{lemma}
Let $X$ be any vector field on $M$. Then%
\label{misc calcs lemma}

\begin{enumerate}
\item The covariant derivative satisfies $\nabla _{\xi } =\mathcal{L}_{\xi } -\sum_{k}e^{k}\wedge
\left( \phi e_{k}\right) \lrcorner $, where $\{\xi ,e_{k}\}$ is a local orthonormal frame of $TM$ with corresponding
coframe $\{\eta ,e^{k}\}$.



\item We have  $%
[X\lrcorner ,L]=(\phi X)^{\flat }\wedge $ and $[X\lrcorner ,\Lambda ]=0$ as well as $[L,\nabla_X]=X^\flat\wedge\eta\wedge$.
\end{enumerate}
\end{lemma}

\begin{proof} Take any differential form $\omega$ and compute
\begin{eqnarray*}
\mathcal{L}_{\xi }\omega  &=&d\left( \xi \lrcorner \omega \right) +\xi
\lrcorner d\omega  \\
&=&\eta \wedge \nabla _{\xi }\left( \xi \lrcorner \omega \right)
+\sum_{k}e^{k}\wedge \nabla _{e_{k}}\left( \xi \lrcorner \omega \right) +\xi
\lrcorner \left(\eta \wedge \nabla _{\xi } \omega \right) +\sum_{k}\xi
\lrcorner \left( e^{k}\wedge \nabla _{e_{k}}\omega \right)  \\
&=&\eta \wedge \xi \lrcorner \nabla _{\xi }\omega +\sum_{k}e^{k}\wedge
\left( \nabla _{e_{k}}\xi \right) \lrcorner \omega +\sum_{k}e^{k}\wedge \xi
\lrcorner \nabla _{e_{k}}\omega+\nabla _{\xi }\omega-\eta\wedge\xi\lrcorner\nabla_\xi\omega\\&& -\sum_{k}
e^{k}\wedge \xi \lrcorner\nabla _{e_{k}}\omega  \\
&=&\nabla _{\xi }\omega +\sum_{k}e^{k}\wedge \left( \phi e_{k}\right)
\lrcorner \omega.
\end{eqnarray*}%
This proves (1). 
Next, if $X$ is any vector field and $\omega $ is any
differential form, we write
$$\lbrack X\lrcorner ,L]\omega =\frac{1}{2} X\lrcorner \left(
d\eta\wedge \omega \right) -\frac{1}{2}  d\eta\wedge X\lrcorner \omega  =\frac{1}{2}(X\lrcorner d\eta)\wedge \omega=(\phi X)^{\flat }\wedge \omega.$$
Also, it is easy to see that $\lbrack X\lrcorner ,\Lambda ]=0$ from the fact that $\Lambda=\frac{1}{2}\sum_{k}\phi
e_{k}\lrcorner e_{k}\lrcorner$. Finally, using that $\nabla_X d\eta=2\eta\wedge X^\flat$, we compute for any differential $\omega$
$$\nabla_X(L\omega)=\frac{1}{2}\nabla_X(d\eta\wedge\omega)=\eta\wedge X^\flat\wedge\omega+\frac{1}{2} d\eta\wedge \nabla_X\omega=\eta\wedge X^\flat\wedge\omega+L(\nabla_X\omega).$$
This gives the required identity.
\end{proof}

We denote by $X^{\pm }=\frac{1}{2}(X\mp i\phi X)$ for all vector fields 
$X\perp \xi $. We have $\phi X^{\pm }=\pm iX^{\pm }$. Hence, for any
differential form $\omega \in \Omega ^{r,s,t}$, we have $(X^{-})^{\flat
}\wedge \omega \in \Omega ^{r+1,s,t}$ and $(X^{+})^{\flat }\wedge \omega \in
\Omega ^{r,s+1,t}$, where $(X^{\pm })^{\flat }:=\frac{1}{2}(X^{\flat }\mp
i(\phi X)^{\flat })$. Also we have that $X^{+}\lrcorner \omega \in \Omega
^{r-1,s,t}$ and $X^{-}\lrcorner \omega \in \Omega ^{r,s-1,t}$. We may write
any differential form $\omega $ as $\omega =\omega _{0}+\eta \wedge (\xi
\lrcorner \omega )$ with $\xi \lrcorner \omega _{0}=0$; this decomposes $%
\omega $ into its \emph{horizontal} part $\omega _{0}=\xi \lrcorner (\eta
\wedge \omega )$ and \emph{vertical} part $\eta \wedge (\xi \lrcorner \omega
)$. 
In order to find the local expression of the differentials $\partial$ and $\overline\partial$ on all forms, we consider a
local orthonormal frame $\{\xi \}\cup \{e_{k}\}$ of $TM$ and corresponding
coframe $\{\eta \}\cup \{e^{k}\}$. Using the first identity in Lemma \ref{misc calcs lemma}, we write for any differential form $\omega $  
\begin{eqnarray*}
d\omega  &=&\eta \wedge \nabla _{\xi }\omega +\sum_{k}e^{k}\wedge \nabla
_{e_{k}}\omega  \\
&=&\eta \wedge (\mathcal{L}_{\xi }\omega -\sum_{k}e^{k}\wedge \phi
e_{k}\lrcorner \omega )+\sum_{k}e^{k}\wedge \nabla _{e_{k}}\omega  \\
&=&d^{0,0,1}\omega -\sum_{k}\eta \wedge e^{k}\wedge \phi e_{k}\lrcorner
\omega +\sum_{k}e^{k}\wedge \nabla _{e_{k}}(\omega _{0}+\eta \wedge \xi
\lrcorner \omega ) \\
&=&d^{0,0,1}\omega -\sum_{k}\eta \wedge e^{k}\wedge \phi e_{k}\lrcorner
\omega +\sum_{k}e^{k}\wedge \nabla _{e_{k}}\omega _{0}+\sum_{k}e^{k}\wedge
(\phi e_{k})^{\flat }\wedge \xi \lrcorner \omega  \\
&&+\sum_{k}e^{k}\wedge \eta \wedge \nabla _{e_{k}}(\xi \lrcorner \omega ) \\
&=&d^{0,0,1}\omega +d^{1,1,-1}\omega -\sum_{k}\eta \wedge e^{k}\wedge \phi
e_{k}\lrcorner \omega +\sum_{k}e^{k}\wedge \nabla _{e_{k}}\omega
_{0}-\sum_{k}\eta \wedge e^{k}\wedge \nabla _{e_{k}}(\xi \lrcorner \omega ),
\end{eqnarray*}%
from which we deduce that 
\begin{equation}
(\partial +\overline{\partial })\omega =\sum_{k}e^{k}\wedge \nabla
_{e_{k}}\omega _{0}-\sum_{k}\eta \wedge e^{k}\wedge \nabla _{e_{k}}(\xi
\lrcorner \omega )-\sum_{k}\eta \wedge e^{k}\wedge \phi e_{k}\lrcorner
\omega .  \label{del delbar formula}
\end{equation}



We now restrict \eqref{del delbar formula} to forms $\omega$ of type $(r,s,t)$ and
observe that we can separate the $(r+1,s,t)$ and $(r,s+1,t)$ parts of the
formula. Given a specific point, we choose a local frame $\{e_{k}\}\cup
\{\xi \}$ such that all covariant derivatives of the transverse frame $%
\{e_{k}\}$ in the horizontal direction are zero at that point (we can do this since the foliation by $%
\xi $ orbits is Riemannian). Then, at the point in question, covariant
derivatives of the type $\nabla _{e_{k}}$ or $\nabla _{e_{k}^{\pm }}$ do not
change the $(r,s,t)$ type of a horizontal form. Next, observe that the $(r+1,s,t)$
part of the formula has the property that $e_{k_{1}}^{-}\lrcorner \ldots
e_{k_{s+1}}^{-}\lrcorner \omega =0$ for all possible choices of $s+1
$ indices, and $\nabla _{e_{\bullet }}$ commutes with those interior
products. We recall that $(e_{k}^{-})^{\flat }$ is of type $(1,0,0)$, so
that by restricting to forms of type $(r,s,t)$, \eqref{del delbar formula}
yields

\begin{eqnarray}
\partial \omega  &=&\Pi ^{r+1,s,t}\left( \sum_{k}e^{k}\wedge \nabla
_{e_{k}}\omega _{0}-\sum_{k}\eta \wedge e^{k}\wedge \nabla _{e_{k}}(\xi
\lrcorner \omega )-\sum_{k}\eta \wedge e^{k}\wedge \phi e_{k}\lrcorner
\omega \right)   \notag \\
&=&\sum_{k}\left( e_{k}^{-}\right) ^{\flat }\wedge \nabla _{e_{k}^{+}}\omega
_{0}-\sum_{k}\eta \wedge \left( e_{k}^{-}\right) ^{\flat }\wedge \nabla
_{e_{k}^{+}}(\xi \lrcorner \omega )-i\sum_{k}\eta \wedge (e_{k}^{-})^{\flat
}\wedge e_{k}^{+}\lrcorner \omega.  \label{del in terms of frame}
\end{eqnarray}%
Similarly, 
\begin{equation}
\overline{\partial }\omega =\sum_{k}(e_{k}^{+})^{\flat }\wedge \nabla
_{e_{k}^{-}}\omega _{0}-\sum_{k}\eta \wedge (e_{k}^{+})^{\flat }\wedge
\nabla _{e_{k}^{-}}(\xi \lrcorner \omega )+i\sum_{k}\eta \wedge
(e_{k}^{+})^{\flat }\wedge e_{k}^{-}\lrcorner \omega.
\label{del bar in terms of on frame}
\end{equation}


Let $*$ denote the Hodge star operator on differential forms and  let $%
``\deg"$ denote the degree of the corresponding differential form. Using $\ast \Pi ^{r,s,t}=\Pi ^{n-s,n-r,1-t}\ast $, we see that 
\begin{eqnarray*}
\ast \partial \ast \left( -1\right) ^{\deg }\Pi ^{r,s,t} &=&\ast \partial
\Pi ^{n-s,n-r,1-t}\ast \left( -1\right) ^{\deg } \\
&=&\ast \Pi ^{n-s+1,n-r,1-t}\partial \ast \left( -1\right) ^{\deg } \\
&=&\Pi ^{r,s-1,t}\ast \partial \ast \left( -1\right) ^{\deg }.
\end{eqnarray*}%
Similarly, 
\begin{eqnarray*}
\ast \overline\partial \ast \left( -1\right) ^{\deg }\Pi ^{r,s,t} &=&\Pi
^{r-1,s,t}\ast \overline\partial \ast \left( -1\right) ^{\deg }, \\
\ast d^{0,0,1}\ast \left( -1\right) ^{\deg }\Pi ^{r,s,t} &=&\Pi
^{r,s,t-1}\ast d^{0,0,1}\ast \left( -1\right) ^{\deg } \\
\ast d^{1,1,-1}\ast \left( -1\right) ^{\deg }\Pi ^{r,s,t} &=&\Pi
^{r-1,s-1,t+1}\ast d^{1,1,-1}\ast \left( -1\right) ^{\deg }.
\end{eqnarray*}
Hence, using the formula $\delta=\ast d\ast (-1)^{\deg}$, it follows that%
\begin{equation}  \label{adj of del in terms of *}
\overline{\partial}^{\ast } =\ast \partial\ast \left( -1\right) ^{\deg }
,\qquad \partial^{\ast } =\ast \overline{\partial}\ast \left( -1\right)
^{\deg }.
\end{equation}
As a consequence, 
\begin{equation}
\overline{\partial _{C}}^{\ast }=\ast \partial _{C}\ast \left( -1\right)
^{\deg },\qquad \partial _{C}^{\ast }=\ast \overline{\partial _{C}}\ast
\left( -1\right) ^{\deg }.  \label{del C adjoint formulas}
\end{equation}


In order to establish identities between $L,\Lambda,\partial,\overline\partial,\partial^\ast, \overline \partial^\ast$ similar to those on K\"ahler manifolds for all differential forms and not just for horizontal forms as done in \cite{thesisStromenger}, we need to introduce the following:
\begin{definition}
We define the twisted differential $d^{c}:\Omega ^{m}\rightarrow
\Omega ^{m+1}$ by%
\begin{equation*}
d^{c} =\sum_{k}(\phi e_{k})^{\flat }\wedge \nabla _{e_{k}}
\end{equation*}%
on any $m$-form, where $\{e_{k}\}$ is a local orthonormal frame
for $\xi ^{\bot }$.
\end{definition}

Note that this definition is independent of the choice of the orthonormal frame,
by the $C^{\infty }(M)$-linearity of $\nabla _{\bullet }$, $\phi \bullet $,
and $(\bullet )^{\flat }$. 

\begin{lemma} \label{expressiondc} On $\Omega^m$, the twisted differential satisfies $d^{c}=i(\overline{\partial }-\partial )+m \eta
\wedge $.
\end{lemma}

\begin{proof}
From \eqref{del in terms of frame} and \eqref{del bar in terms of on frame}
we write for any $\omega =\omega _{0}+\eta \wedge \xi \lrcorner \omega $,%
\begin{eqnarray*}
i(\overline{\partial }-\partial )\omega  &=&i\Big(\sum_{k}(e_{k}^{+})^{\flat
}\wedge \nabla _{e_{k}^{-}}\omega _{0}-\sum_{k}\eta \wedge
(e_{k}^{+})^{\flat }\wedge \nabla _{e_{k}^{-}}(\xi \lrcorner \omega
)+i\sum_{k}\eta \wedge (e_{k}^{+})^{\flat }\wedge e_{k}^{-}\lrcorner \omega
 \\
&&-\sum_{k}(e_{k}^{-})^{\flat }\wedge \nabla _{e_{k}^{+}}\omega
_{0}+\sum_{k}\eta \wedge (e_{k}^{-})^{\flat }\wedge \nabla _{e_{k}^{+}}(\xi
\lrcorner \omega )+i\sum_{k}\eta \wedge (e_{k}^{-})^{\flat }\wedge
e_{k}^{+}\lrcorner \omega \Big) \\
&=&i\Big(\sum_{k}(e_{k}^{+}-e_k^-)^{\flat }\wedge \nabla _{e_{k}}\omega
_{0}-\sum_{k}\eta \wedge (e_{k}^{+}-e_k^-)^{\flat }\wedge \nabla _{e_{k}}(\xi
\lrcorner \omega )\\&&+i\sum_{k}\eta \wedge e^{k}\wedge e_{k}\lrcorner \omega \Big).
\end{eqnarray*}%
Now from $(e_{k}^{+}-e_{k}^{-})^{\flat
}=-i(\phi e_{k})^{\flat }$, the above equation becomes
\begin{eqnarray*}
i(\overline{\partial }-\partial )\omega  &=&\sum_{k}(\phi e_{k})^{\flat
}\wedge \nabla _{e_{k}}\omega _{0}-\sum_{k}\eta \wedge (\phi e_{k})^{\flat
}\wedge \nabla _{e_{k}}(\xi \lrcorner \omega )-\sum_{k}\eta \wedge
e^{k}\wedge e_{k}\lrcorner \omega  \\
&=&d^{c}\omega _{0}-\eta \wedge d^{c}(\xi \lrcorner \omega )-\sum_{k}\eta
\wedge e^{k}\wedge e_{k}\lrcorner \omega \\
&=&d^{c}\omega _{0}-\eta \wedge d^{c}(\xi \lrcorner \omega )-\deg(\omega)\eta
\wedge \omega.
\end{eqnarray*}
Using that $\{d^c,\eta\wedge\}=0$, which be easily proven, we deduce   
$i(\overline{\partial }-\partial )\omega=d^{c}\omega -\deg(\omega)\eta
\wedge  \omega$. This finishes the proof.
\end{proof}

\begin{lemma}
\label{delta c lemma}
The formal adjoint $\delta ^{c}=\left( d^{c}\right) ^{\ast }:\Omega
^{m+1}\rightarrow \Omega ^{m}$ satisfies%
\begin{equation*}
\delta ^{c}=-i(\overline{\partial }^{\ast }-\partial ^{\ast })+m\xi\lrcorner.
\end{equation*}
\end{lemma}

Now, we state the following identities that were proven in \cite[Prop. 1.2.6]{thesisStromenger} for all forms, but we include the proof for completeness.
\begin{lemma}
\label{L delta commutator lemma} On $\Omega^m$, we have %
\begin{eqnarray*}
\lbrack L,\delta ] &=&d^{c}-(2n-m) \eta \wedge  \\
\text{ }[\Lambda ,d] &=&-\delta ^{c}+(2n-m+1) \xi \lrcorner. 
\end{eqnarray*}
\end{lemma}

\begin{proof}
Using $[L,X\lrcorner]=-\phi(X)^\flat\wedge$ and $[L,\nabla_X]=X^\flat\wedge\eta$ from Lemma \ref{misc calcs lemma}, we consider a local orthonormal frame $\{f_k\}=\{\xi,e_k\}$ of $TM$ and write for any differential form $%
\omega\in \Omega^m $,%
\begin{eqnarray*}
L\delta\omega  &=&-\sum_{k} L(f_{k}\lrcorner \nabla
_{f_{k}}\omega)\\
&=&-\sum_{k} f_{k}\lrcorner L(\nabla _{f_{k}}\omega)+(\phi(f_k))^\flat\wedge \nabla_{f_k}\omega\\
&=&\delta L\omega- \sum_k f_k\lrcorner (f_k^\flat\wedge\eta\wedge\omega)+d^c\omega\\
&=&\delta L\omega+(m-2n)\eta\wedge\omega+d^c\omega
\end{eqnarray*}%
which is the first formula. The second formula follows by taking adjoints.
\end{proof}
We use the previous computations to state the analogues of K\"ahler identities on all forms.
\begin{proposition}\label{Kaehler Identity prop}
(Analogues of K\"{a}hler identities on Sasakian manifolds) On the Sasakian
manifold $(M^{2n+1},g,\xi )$, we have the following identities \nopagebreak 
\begin{eqnarray}
\left[ \Lambda ,\overline{\partial }\right]  &=&-i\partial ^{\ast },
\label{eq:kaehler ids} \\
\left[ \Lambda ,\partial \right]  &=&i\overline{\partial }^{\ast },  \notag
\\
\left[ \overline{\partial }^{\ast },L\right]  &=&i\partial ,  \notag \\
\left[ \partial ^{\ast },L\right]  &=&-i\overline{\partial }.  \notag
\end{eqnarray}
\end{proposition}

\begin{proof}
Using Lemmas \ref{delta c lemma} and \ref{L delta
commutator lemma}, we compute for any $\omega \in \Omega ^{r,s,t}$ with $m=r+s+t$, 
\begin{eqnarray*}
\left[ \Lambda ,\overline{\partial }\right] \omega  &=&\Pi ^{r-1,s,t}\left[
\Lambda ,d\right] \omega =\Pi ^{r-1,s,t}
(-\delta ^{c}+(2n-m+1) \xi \lrcorner)
\omega   \\
&=&\Pi ^{r-1,s,t}\left( 
i(\overline{\partial }^{\ast }-\partial ^{\ast })-(m-1)\xi\lrcorner
+(2n-m+1) \xi \lrcorner
\right) \omega =-i\partial ^{\ast }\omega .
\end{eqnarray*}%
The other identities follow by taking adjoints and conjugates.
\end{proof}

\begin{corollary}
We have%
\begin{equation}\label{eq:partialstarpartialcom}
\left\{ \partial ,\overline{\partial }^{\ast }\right\} =\left\{ \overline{%
\partial },\partial ^{\ast }\right\} =0.
\end{equation}
\end{corollary}

\begin{proof}
From the proposition above, we write
\begin{eqnarray*}
\left\{ \partial ,\overline{\partial }^{\ast }\right\}  &=&-i\left\{
\partial ,\left[ \Lambda ,\partial \right] \right\} =-i(\partial (\Lambda
\partial -\partial \Lambda )+(\Lambda \partial -\partial \Lambda )\partial )
\\
&=&-i(\partial \Lambda \partial -\partial \Lambda \partial )=0.
\end{eqnarray*}
\end{proof}
In the following, we state similar identities for the differentials $\partial_C=\partial+L\left( \xi \lrcorner \right)$ and $\overline{\partial 
}_{C}=\overline{\partial }+L\left( \xi \lrcorner \right)$. For this, we define $H$ as 
\begin{equation}
H:=\left[ \Lambda ,L\right] =(n-\mathrm{deg})+\eta \wedge (\xi
\lrcorner ).  \label{eq:comm lambda L}
\end{equation}%
An easy computation shows that $\left[ H,L\right] =-2L$ and $\left[ H,\Lambda %
\right] =2\Lambda $. Note that $H,L,\Lambda $ all commute with $\eta \wedge $
and $\xi \lrcorner $ and of course with $\mathcal{L}_{\xi }$. We  now apply these identities to prove the following:

\begin{proposition}
\label{del bar C commutator formulas} (formulas stated in \cite%
{Schmude2014LaplaceOpsSasakiEinstMflds} for Sasakian-Einstein case) The
following identities hold on a Sasakian manifold $(M^{2n+1},g,\xi )$
\begin{enumerate}
\item $\left\{ \partial _{C},\xi \lrcorner \right\} =\left\{ \overline{%
\partial _{C}},\xi \lrcorner \right\} =0,\left\{ \partial _{C}^{\ast },\eta
\wedge \right\} =\left\{ \overline{\partial _{C}}^{\ast },\eta \wedge
\right\} =0$

\item $\left\{ \partial _{C},\eta \wedge \right\} =\left\{ \overline{%
\partial _{C}},\eta \wedge \right\} =L,\,\,\left\{ \partial _{C}^{\ast },\xi
\lrcorner \right\} =\left\{ \overline{\partial _{C}}^{\ast },\xi \lrcorner
\right\} =\Lambda $

\item $\overline{\partial _{C}}\partial _{C}+\partial _{C}\overline{\partial
_{C}}=-2L\mathcal{L}_{\xi }$

\item $\overline{\partial _{C}}^{\ast }=\ast \partial _{C}\ast \left(
-1\right) ^{\deg },\qquad \partial _{C}^{\ast }=\ast \overline{\partial _{C}}%
\ast \left( -1\right) ^{\deg }$

\item $0=[L,\partial _{C}]=[L,\overline{\partial _{C}}]=[\Lambda ,\overline{%
\partial _{C}}^{\ast }]=[\Lambda ,\partial _{C}^{\ast }]$

\item $\left[ H,\partial _{C}\right] =-\partial _{C}-L(\xi \lrcorner )$

\item $[\Lambda ,\overline{\partial _{C}}]=-i\partial _{C}^{\ast }+i\eta
\wedge \Lambda +H(\xi \lrcorner )$

\item $[\Lambda ,\partial _{C}]=i\overline{\partial _{C}}^{\ast }-i\eta
\wedge \Lambda +H(\xi \lrcorner )$

\item $[\overline{\partial _{C}}^{\ast },L]=i\partial _{C}-iL(\xi \lrcorner
)+\eta \wedge H$

\item $[\partial _{C}^{\ast },L]=-i\overline{\partial _{C}}+iL(\xi \lrcorner
)+\eta \wedge H$.
\end{enumerate}
\end{proposition}

\begin{proof}
(1) follows from $\{\partial ,\xi \lrcorner \}=[\partial ,L]=0$. (2) follows
from $\{\partial ,\eta \wedge \}=0$ and $\{L\left( \xi \lrcorner \right)
,\eta \wedge \}=L$. From $\{\partial ,\overline{\partial }\}=-2L\mathcal{L}%
_{\xi }$, we have 
\begin{eqnarray*}
\{\partial +L\left( \xi \lrcorner \right) ,\overline{\partial }+L\left( \xi
\lrcorner \right) \} &=&\{\partial ,\overline{\partial }\}+(\partial +%
\overline{\partial })L\left( \xi \lrcorner \right) +L\xi \lrcorner \left( 
\overline{\partial }+\partial \right) \\
&=&-2L\mathcal{L}_{\xi }+L(\partial +\overline{\partial })\left( \xi
\lrcorner \right) +L\xi \lrcorner \left( \overline{\partial }+\partial
\right) \\
&=&-2L\mathcal{L}_{\xi }+L\{\partial,\xi\lrcorner\}+L\{\overline\partial,\xi\lrcorner\}=-2L\mathcal{L}_{\xi },
\end{eqnarray*}%
proving (3). Equation (4) follows from the formulas for $\partial ^{\ast }$ and $%
\Lambda $ in terms of the Hodge star. (5) follows from $[L,\partial ]=[L,\xi
\lrcorner ]=0$. Next, from the fact that $[H,\partial ]=-\partial$ which can be proven using the identities from Proposition \ref{Kaehler Identity prop}, we compute
$$\left[ H,\partial _{C}\right]=[H,\partial +L\xi \lrcorner ]=[H,\partial
]+HL\xi \lrcorner -L\xi \lrcorner H 
=-\partial -2L\xi \lrcorner=-\partial _{C}-L(\xi \lrcorner ),$$
yielding (6). At last, 
$$\lbrack \Lambda ,\overline{\partial _{C}}] =[\Lambda ,\overline{\partial }%
]+\Lambda L\xi \lrcorner -L\xi \lrcorner \Lambda =-i\partial ^{\ast }+\left[
\Lambda ,L\right] \xi \lrcorner =-i\partial _{C}^{\ast }+i\eta \wedge \Lambda +H\xi \lrcorner ,$$
proving (7). The identities (8), (9), and (10) follow by conjugating and/or
taking adjoints.
\end{proof}

\section{Laplacians on Sasakian manifolds}
\label{Sasaki section}

We define the Dolbeault Laplace operators associated with $\partial $
and $\overline{\partial }$ by the following: 
\begin{equation*}
\Delta _{\partial }=\partial \partial ^{\ast }+\partial ^{\ast }\partial
,~\Delta _{\overline{\partial }}=\overline{\partial }\overline{\partial }%
^{\ast }+\overline{\partial }^{\ast }\overline{\partial }.
\end{equation*}

We now prove Theorem \ref{Laplacian theorem} stated in the introduction.  

\begin{proof}[Proof of Theorem \ref{Laplacian theorem}] Using the identities in Proposition \ref{Kaehler Identity prop}, we compute 
\begin{eqnarray*}
\Delta _{\partial } &=&\partial ^{\ast }\partial +\partial \partial ^{\ast
}=i\left[ \Lambda ,\overline{\partial }\right] \partial +i\partial \left[
\Lambda ,\overline{\partial }\right]  \\
&=&i\Lambda \overline{\partial }\partial -i\overline{\partial }\Lambda
\partial +i\partial \Lambda \overline{\partial }-i\partial \overline{%
\partial }\Lambda  \\
&=&i\Lambda \overline{\partial }\partial -i\overline{\partial }\left( \left[
\Lambda ,\partial \right] +\partial \Lambda \right) +i\left( \left[ \partial
,\Lambda \right] +\Lambda \partial \right) \overline{\partial }-i\partial 
\overline{\partial }\Lambda  \\
&=&i\Lambda \overline{\partial }\partial -i\overline{\partial }\left( i%
\overline{\partial }^{\ast }+\partial \Lambda \right) +i\left( -i\overline{%
\partial }^{\ast }+\Lambda \partial \right) \overline{\partial }-i\partial 
\overline{\partial }\Lambda  \\
&=&\overline{\partial }\overline{\partial }^{\ast }+\overline{\partial }%
^{\ast }\overline{\partial }-i\left( \overline{\partial }\partial +\partial 
\overline{\partial }\right) \Lambda +i\Lambda \left( \overline{\partial }%
\partial +\partial \overline{\partial}\right).
\end{eqnarray*}%
Since $\overline{\partial }\partial +\partial \overline{\partial }=-2L%
\mathcal{L}_{\xi }$ and $\left[ \Lambda ,L\right] =H,$ we get%
\begin{eqnarray*}
\Delta _{\partial } &=&\Delta _{\overline{\partial }}+2iL\mathcal{L}_{\xi
}\Lambda -2i\Lambda L\mathcal{L}_{\xi } \\
&=&\Delta _{\overline{\partial }}+2i\left[ L,\Lambda \right] \mathcal{L}%
_{\xi } \\
&=&\Delta _{\overline{\partial }}-2iH\mathcal{L}_{\xi },
\end{eqnarray*}%
which is the required formula. Next, we want to express the Hodge Laplacian in terms of $\Delta_{\overline\partial}$. Since $\left[ \Lambda , %
\overline{\partial }\right] =-i\partial ^{\ast }$, we write
$$\left\{ \overline{\partial },\partial ^{\ast }\right\}  =\left\{ \overline{%
\partial },i\left[ \Lambda ,\overline{\partial }\right] \right\}  
=i\overline{\partial }\Lambda \overline{\partial }-i\overline{\partial }%
^{2}\Lambda +i\Lambda \overline{\partial }^{2}-i\overline{\partial }\Lambda 
\overline{\partial }=0,$$
and likewise $\left\{ \partial ,\overline{\partial }^{\ast }\right\} =0$.
Then%
\begin{equation*}
\left\{ \left( \partial +\overline{\partial }\right) ,\left( \partial +%
\overline{\partial }\right) ^{\ast }\right\} =\Delta _{\partial }+\Delta _{%
\overline{\partial }}.
\end{equation*}%
Thus, as $d^{1,1,-1}=2L(\xi \lrcorner )$ and $d^{0,0,1}=\eta
\wedge \mathcal{L}_{\xi }$, we have
\begin{eqnarray*}
\Delta  &=&\left\{ \partial +\overline{\partial }+d^{1,1,-1}+d^{0,0,1},%
\left( \partial +\overline{\partial }+d^{1,1,-1}+d^{0,0,1}\right) ^{\ast
}\right\}  \\
&=&\left\{ \partial +\overline{\partial },\partial ^{\ast }+\overline{%
\partial }^{\ast }\right\} +\{d^{1,1,-1}+d^{0,0,1},\partial ^{\ast }+%
\overline{\partial }^{\ast }\} \\
&&+\left\{ \delta ^{-1,-1,1}+\delta ^{0,0,-1},\partial +\overline{\partial }%
\right\} +\left\{ d^{1,1,-1}+d^{0,0,1},\delta ^{-1,-1,1}+\delta
^{0,-1,0}\right\}. 
\end{eqnarray*}%
We consider%
\begin{eqnarray*}
\left\{ \delta ^{-1,-1,1}+\delta ^{0,0,-1},\partial \right\}  &=&2\eta
\wedge \Lambda \partial -\mathcal{L}_{\xi }(\xi \lrcorner \partial
)+2\partial (\eta \wedge \Lambda )-\partial \mathcal{L}_{\xi }(\xi \lrcorner
) \\
&=&2\eta \wedge \left[ \Lambda ,\partial \right] -\mathcal{L}_{\xi }\left\{
\xi \lrcorner ,\partial \right\} .
\end{eqnarray*}%
Using $\left\{ \xi \lrcorner ,\partial \right\} =0$ and $\left[ \Lambda
,\partial \right] =i\overline{\partial }^{\ast }$, we get that 
\begin{equation*}
\left\{ \delta ^{-1,-1,1}+\delta ^{0,0,-1},\partial \right\} =2i\eta \wedge 
\overline{\partial }^{\ast }.
\end{equation*}%
By taking conjugates and adding,%
\begin{equation*}
\left\{ \delta ^{-1,-1,1}+\delta ^{0,0,-1},\partial +\overline{\partial }%
\right\} =2i\eta \wedge \left( \overline{\partial }^{\ast }-\partial ^{\ast
}\right) .
\end{equation*}%
The adjoint is%
\begin{eqnarray*}
\left\{ d^{1,1,-1}+d^{0,0,1},\partial ^{\ast }+\overline{\partial }^{\ast
}\right\}  &=&-2i\left( \overline{\partial }-\partial \right) \xi \lrcorner 
\\
&=&2i\xi \lrcorner \left( \overline{\partial }-\partial \right) .
\end{eqnarray*}%
Also, since $\xi \lrcorner $ and $\eta \wedge $ commute with $L$ and $%
\Lambda $, and since $\mathcal{L}_{\xi }$ commutes with all four of the
other operators, we have
\begin{eqnarray*}
\left\{ d^{1,1,-1}+d^{0,0,1},\delta ^{-1,-1,1}+\delta ^{0,0,-1}\right\} 
&=&4\left\{ L(\xi \lrcorner ),\eta \wedge \Lambda \right\} -2\left\{ L(\xi
\lrcorner ),\mathcal{L}_{\xi }(\xi \lrcorner )\right\}  \\
&&+2\left\{ \eta \wedge \mathcal{L}_{\xi },\eta \wedge \Lambda \right\}
-\left\{ \eta \wedge \mathcal{L}_{\xi },\mathcal{L}_{\xi }(\xi \lrcorner
)\right\}  \\
&=&4\xi \lrcorner (\eta \wedge L\Lambda )+4\eta \wedge \xi \lrcorner \Lambda
L-\mathcal{L}_{\xi }^{2} \\
&=&4L\Lambda +4\eta \wedge \xi \lrcorner \lbrack \Lambda ,L]-\mathcal{L}%
_{\xi }^{2} \\
&=&4L\Lambda +4\eta \wedge \xi \lrcorner H-\mathcal{L}_{\xi }^{2}.
\end{eqnarray*}%
Thus, plugging the above computations yield%
\begin{eqnarray*}
\Delta  &=&\Delta _{\partial }+\Delta _{\overline{\partial }}+2i\eta \wedge
\left( \overline{\partial }^{\ast }-\partial ^{\ast }\right) +2i\xi
\lrcorner \left( \overline{\partial }-\partial \right) +4L\Lambda +4\eta
\wedge \xi \lrcorner H-\mathcal{L}_{\xi }^{2} \\
&=&2\Delta _{\overline{\partial }}-2iH\mathcal{L}_{\xi }+2i\eta \wedge
\left( \overline{\partial }^{\ast }-\partial ^{\ast }\right) +2i\xi
\lrcorner \left( \overline{\partial }-\partial \right) +4L\Lambda +4\eta
\wedge \xi \lrcorner H-\mathcal{L}_{\xi }^{2},
\end{eqnarray*}
which is Equality \eqref{eq:laplaceDel}. 
\end{proof}



We define the Dolbeault Laplacian operator associated with $\partial _{C}$ and $%
\overline{\partial }_{C}$ by the following: 
\begin{equation*}
\Delta _{\partial _{C}}=\partial _{C}\partial _{C}^{\ast }+\partial
_{C}^{\ast }\partial _{C},~~\Delta _{\overline{\partial _{C}}}=\overline{%
\partial }_{C}\overline{\partial }_{C}^{\ast }+\overline{\partial }%
_{C}^{\ast }\overline{\partial }_{C}.~~
\end{equation*}
As before, we will find a relation between these Dolbeault Laplacian operators with the Hodge Laplacian. These formulas were stated in \cite{Schmude2014LaplaceOpsSasakiEinstMflds} for Sasaki-Einstein manifolds. We have
\begin{theorem}\label{Laplacian C theorem}
(formulas stated in \cite{Schmude2014LaplaceOpsSasakiEinstMflds} for
Sasakian-Einstein case) Let $(M^{2n+1},g,\xi )$ be a Sasakian manifold. The Dolbeault Laplacians $\Delta_{\partial _{C}}$ and $%
\Delta_{\overline{\partial _{C}}}$ are related by 
\begin{equation}  \label{eq:deltacdeltabarc}
\Delta _{\partial _{C}}=\Delta _{\overline{\partial _{C}}}+i\eta\wedge(%
\overline{\partial _{C}}^{\ast }+\partial _{C}^{\ast })-i\left( \overline{%
\partial _{C}}+\partial _{C}\right) \xi\lrcorner-2iH\mathcal{L}_{\xi }.
\end{equation}%
The Hodge Laplacian on $M$ is expressed by the formula 
\begin{equation}  \label{eq:deltahodgedeltac}
\Delta=2\Delta _{\overline{\partial _{C}}}-\mathcal{L}_{\xi }^{2}-2iH%
\mathcal{L}_{\xi }+2L\Lambda +2H\eta\wedge (\xi\lrcorner)+2i\left( \eta\wedge%
\overline{\partial _{C}}^{\ast }-\overline{\partial _{C}}\xi\lrcorner\right).
\end{equation}
\end{theorem}

\begin{proof}
We first prove the identity 
\begin{equation}\label{eq:relationdeltabardelta}
\Delta_{\overline{\partial_C}}=\Delta_{\overline{\partial}%
}-i\eta\wedge\partial^*-i\xi\lrcorner\partial+L
\Lambda+H\eta\wedge\xi\lrcorner.
\end{equation}
Indeed, we compute%
\begin{eqnarray*}
\Delta _{\overline{\partial _{C}}} &=&\overline{\partial _{C}}^{\ast }%
\overline{\partial _{C}}+\overline{\partial _{C}}\overline{\partial _{C}}%
^{\ast } \\
&=&\left( \overline{\partial }^{\ast }+\eta \wedge \Lambda \right) \left( 
\overline{\partial }+L\xi \lrcorner \right) +\left( \overline{\partial }%
+L\xi \lrcorner \right) \left( \overline{\partial }^{\ast }+\eta \wedge
\Lambda \right)  \\
&=&\Delta _{\overline{\partial }}+\left\{ \eta \wedge \Lambda ,\overline{%
\partial }\right\} +\left\{ \overline{\partial }^{\ast },L\xi \lrcorner
\right\} +\eta \wedge \Lambda L\xi \lrcorner +L\xi \lrcorner \eta \wedge
\Lambda  \\
&=&\Delta _{\overline{\partial }}+\eta \wedge \lbrack \Lambda ,\overline{%
\partial }]+\xi \lrcorner \lbrack L,\overline{\partial }^{\ast }]+\Lambda
L\eta \wedge \xi \lrcorner +L\Lambda \xi \lrcorner \eta \wedge  \\
&=&\Delta _{\overline{\partial }}-i\eta \wedge \partial^\ast-i\xi \lrcorner \partial+L\Lambda
+H\eta \wedge \xi \lrcorner ,
\end{eqnarray*}%
using $[\Lambda ,L]=H$ and $\left\{ \eta \wedge ,\xi \lrcorner \right\} =1$ and the K\"{a}hler identities from Proposition \ref{Kaehler Identity prop}. 
Taking now the conjugate of \eqref{eq:relationdeltabardelta} and using Equality \eqref{laplaceDels}

\begin{eqnarray*}
\Delta_{\partial_C}&=&\Delta_{\partial}%
+i\eta\wedge\overline\partial^*+i\xi\lrcorner\overline\partial+L
\Lambda+H\eta\wedge\xi\lrcorner\\
&=&\Delta_{\overline\partial}-2iH\mathcal{L}_\xi+i\eta\wedge\overline\partial^*+i\xi\lrcorner\overline\partial+L
\Lambda+H\eta\wedge\xi\lrcorner\\
&=&\Delta_{\overline\partial_C}+i\eta\wedge(\partial^\ast+\overline\partial^*)+i\xi\lrcorner(\partial+\overline\partial)-2iH\mathcal{L}_\xi\\
&=&\Delta _{\overline{\partial _{C}}}+i\eta\wedge(%
\overline{\partial _{C}}^{\ast }+\partial _{C}^{\ast })-i\left( \overline{%
\partial _{C}}+\partial _{C}\right) \xi\lrcorner-2iH\mathcal{L}_{\xi },
\end{eqnarray*}
which is Equality \eqref{eq:deltacdeltabarc}. To prove \eqref{eq:deltahodgedeltac}, we use Equation \eqref{eq:laplaceDel}
\begin{eqnarray*}
\Delta &=&2\Delta _{\overline{\partial }}-2iH\mathcal{L}_{\xi }+2i\eta \wedge
\left( \overline{\partial }^{\ast }-\partial ^{\ast }\right) +2i\xi
\lrcorner \left( \overline{\partial }-\partial \right) +4L\Lambda +4\eta
\wedge \xi \lrcorner H-\mathcal{L}_{\xi }^{2}\\
&\stackrel{\eqref{eq:relationdeltabardelta}}{=}&2\Delta_{\overline{\partial_C}}+2L
\Lambda+2H\eta\wedge\xi\lrcorner-2iH\mathcal{L}_{\xi }+2i\eta \wedge
\overline{\partial }^{\ast } +2i\xi
\lrcorner \overline{\partial }-\mathcal{L}_{\xi }^{2}\\
&=&2\Delta_{\overline{\partial_C}}+2L
\Lambda+2H\eta\wedge\xi\lrcorner-2iH\mathcal{L}_{\xi }+2i\eta \wedge
\overline{\partial }_C^{\ast }-2i\overline{\partial }_C(\xi
\lrcorner) -\mathcal{L}_{\xi }^{2},
\end{eqnarray*}
to get the required identity. 
\end{proof}

\section{Spectrum and de Rham-Hodge theory for the Dolbeault Laplacians}
\label{Spectral theory section}

In this section, we let the operator $D$ to be one of the differentials $%
\partial$, $\partial_C$, $\overline{\partial}$, or $\overline{\partial_C}$, and
let $D^*$ be its formal adjoint on the space of all differential forms. We denote by $H_D^*(M)$ the corresponding cohomology group and by $\Delta_D=D^*D+DD^*$ the corresponding Laplacian. We wish to relate the kernel of $\Delta_D$
to the cohomology group $H_D^*(M)$ as in the case of standard de Rham-Hodge
theory.

Note that this task is not trivial, since we are considering unbounded operators on the space of all differential forms, and they
are not elliptic. We comment that in the important papers \cite%
{KohnRossi1965ExtHolomFcnsToBndryCpxMfd}, \cite{Kohn1964BdriesCpxMflds} (see
also \cite{Tanaka1975DiffGeomStudy}, \cite%
{BarlettaDragomirShahid2025FoundCRGeom} for more modern notation), the de
Rham-Hodge theory needed for the study of strictly pseudoconvex CR-manifolds
was worked out; note that Sasakian manifolds are particular examples of these.
In these works, the authors considered the Kohn-Rossi cohomology of
horizontal differential forms (those forms $\alpha $ for which $\xi
\lrcorner \alpha =0$). Restricting to such forms, the operators $\overline{%
\partial }$ and $\overline{\partial _{C}}$ are both equal to their
restriction $\overline{\partial}_{KR}$ and square to zero, yielding respectively the
Kohn-Rossi cohomology groups, which we denote as  $H_{\overline{\partial}_{KR}%
}^{r,s}(M)$. One may define the adjoint $\overline{\partial}_{KR}^{\ast }$
of $\overline{\partial}_{KR}$ with respect to the $L^{2}$-inner product on
the space of horizontal differential forms and then produce the subelliptic
Laplacian $\square _{KR}$. Then, for horizontal differential forms of degree $%
(r,s)$ with $1\leq s\leq n-1$ (called \emph{tangential forms} in the CR\
manifold world), each Kohn-Rossi cohomology class contains a unique $\square
_{KR}$-harmonic form, in fact the form with the smallest $L^{2}$-norm. It
turns out that in those degrees the Kohn-Rossi cohomology is
finite dimensional. For $s=0$, the cohomology is typically
infinite dimensional, as it contains the CR functions.

\begin{proof}[Proof of Theorem \ref{rst Hodge Theorem}] The operators $\overline{\partial}$ and $\overline{\partial_C}$ act on the
space of all differential forms, and in fact they are not equal as
operators on this larger space. To work out the de Rham-Hodge theory, we
partition the space of all differential forms using the action of $%
\mathcal{L}_\xi$. The flow of the vector field $\xi$ produces a
one-parameter family of isometries, whose closure is a (compact) torus $T$
of isometries on $M$. These induce an isometric action of $h\in T$
on differential forms that extends the action of $\mathcal{L}_\xi$ to all of 
$T$. Since $\mathcal{L}_\xi$ commutes with each $\Delta_D$, so does the
action of $T$, and thus all eigenspaces of these operators can be decomposed
into irreducible representation spaces of $T$, which are all
one-dimensional. These representation spaces of differential forms are
specific eigenspaces of the Casimir operator $C_T$ associated with $T$; the
possible eigenvalues $|\lambda|^2$ are determined by the weight lattice of $%
T $.

The operators $\Delta _{D}$ on differential forms of degree $m$ are not
elliptic; however, we have seen in \eqref{eq:laplaceDel} and %
\eqref{eq:deltahodgedeltac} that the symmetric
operators 
\begin{equation*}
\Delta _{D}-\frac{1}{2}\mathcal{L}_{\xi }^{2}
\end{equation*}%
have the same principal symbol as half of the Laplacian (see Equations %
\eqref{eq:deltahodgedeltac} and \eqref{eq:laplaceDel} and their conjugates),
and thus these modified operators are elliptic.
Note that in \cite{GNT2016Rigidity}, the authors made a similar observation in proving that the basic Laplacian on a Sasakian manifold is the restriction of a strongly elliptic operator on all differential forms.

Each eigenspace of the (ordinary) Laplacian on a given Sasakian manifold $M$
of dimension $2n+1$ is finite dimensional and can be orthogonally decomposed
into eigenspaces of $C_{T}$, which in turn can be orthogonally decomposed
into (pure imaginary, finite dimensional) eigenspaces of $\mathcal{L}_{\xi }$%
. In fact, the entire space of $L^{2}$-differential forms over $M$ can be
orthogonally decomposed into eigenspaces of $\mathcal{L}_{\xi }$. We say
that an eigenform $\alpha $ of $\mathcal{L}_{\xi }$ has \textbf{charge }$q$
if $\mathcal{L}_{\xi }\alpha =iq\alpha $; this terminology comes from the
physics literature.

In addition, since $\mathcal{L}_\xi$ and $C_T$ commute with all of the
operators $\Delta_D$ and $\Delta_D-\frac 12\mathcal{L}_\xi^2$, the
eigenspaces of these operators can also be orthogonally decomposed into
eigenspaces of $\mathcal{L}_\xi$; however, the eigenspaces of $\Delta_D$
need not be finite dimensional a priori. What is true is that when
restricting to the spaces of differential forms of degree $m$ and charge $q$%
, the eigenspaces of $\Delta_D$ corresponding to eigenvalue $\lambda$, which
are the eigenspaces of $\Delta_D-\frac 12\mathcal{L}_\xi^2$ corresponding to
eigenvalue $\lambda+\frac 12 q^2$, are finite dimensional. Thus the spectrum
of each $\Delta_D$ is discrete and consists of eigenvalues, and each
eigenspace is orthogonally decomposed as a direct sum of the finite
dimensional subspaces of eigenforms with a specific charge.

Since the second order operator $\Delta _{D}-\frac{1}{2}\mathcal{L}_{\xi
}^{2}-\frac{1}{2}q^{2}$ is symmetric and elliptic, it follows from standard
Hodge theory that the space $\Omega ^{m}$ of differential $m$-forms
satisfies 
\begin{equation*}
\Omega ^{m}=\ker \left( \Delta _{D}-\frac{1}{2}\mathcal{L}_{\xi }^{2}-%
\frac{1}{2}q^{2}\right) ^{(m)}\oplus \func{im}\left( \Delta _{D}-\frac{1}{2}%
\mathcal{L}_{\xi }^{2}-\frac{1}{2}q^{2}\right) ^{(m)},
\end{equation*}%
where the superscript $(m)$ refers to the restriction to $m$-forms, and
ellipticity implies that the kernels consist of smooth differential forms.
This is an orthogonal direct sum, and since $\mathcal{L}_{\xi }$ commutes
with relevant operators, we intersect with the differential forms of charge $%
q$ to obtain the direct sum decomposition 
\begin{eqnarray*}
\Omega ^{m}(q) &=&\ker \left( \Delta _{D}-\frac{1}{2}\mathcal{L}_{\xi
}^{2}-\frac{1}{2}q^{2}\right) ^{(m)}(q)\oplus \func{im}\left( \Delta _{D}-%
\frac{1}{2}\mathcal{L}_{\xi }^{2}-\frac{1}{2}q^{2}\right) ^{(m)}(q) \\
&=&\ker \left( \Delta _{D}^{(m)}(q)\right) \oplus \func{im}\left( \Delta
_{D}^{(m)}(q)\right)  \\
&=&\ker \left( \Delta _{D}^{(m)}(q)\right) \oplus \func{im}%
(D^{(m-1)}(q))\oplus \func{im}(D^{\ast (m+1)}(q))
\end{eqnarray*}%
where $(q)$ indicates the restriction to forms of charge $q$. Then 
\begin{equation}
H_{D}^{m}(q)\cong \ker (\Delta _{D}^{(m)}(q))=\ker (D^{(m)}(q))\cap
\ker (D^{\ast (m)}(q))  \label{Hodge q}
\end{equation}%
proves that each $H_{D}^{m}(q)$ is finite dimensional. If the torus $T$
is $1$-dimensional, i.e. if the orbits of $\xi $ are closed, then the
character of each irreducible $T$ representation is $e^{2\pi iax}$ for some $%
a\in \mathbb{Z}$, and the Lie derivative acts with charge $2\pi at$ for
some fixed $t\in \mathbb{R}$ and $a\in \mathbb{Z}$. If the torus $T$ is $b$%
-dimensional with $b>1$, the irreducible representations have the characters 
$e^{2\pi ia\cdot x}$ with $a\in \mathbb{Z}^{b}$, and this implies that the
action of $\mathcal{L}_{\xi }$ on such a representation subspace is one with
charge $q=2\pi a\cdot \tau $, where the fixed vector $\tau \in \mathbb{R}^{b}
$ generates a dense subgroup of $\mathbb{R}^{b}\slash\mathbb{Z}^{b}$. Thus
the set of possible charges $q$ such that $H_{D}^{(m)}(q)\neq 0$ is a
subset of $\{2\pi a\cdot \tau :a\in \mathbb{Z}^{m}\}$, which is discrete if $%
m=1$ and otherwise is a countable dense subset of $\mathbb{R}$. Note that in
all cases, the multiplicity of the $e^{2\pi ia\cdot x}$ representation (i.e.
charge $q=2\pi a\cdot \tau $) in $H_{D}^{(m)}(M)$ is finite, since the
kernel of $\Delta _{D}-\frac{1}{2}\mathcal{L}_{\xi }^{2}-\frac{1}{2}q^{2}$
is finite dimensional. Notice also that given any differential $m$-form $%
\alpha $, the projection of this form to the set of forms of charge $q=2\pi
a\cdot \tau $ is $\int_{T}h\cdot \alpha \,e^{-2\pi a\cdot h}\,dh$, where $dh$
is the Haar measure (i.e. Lebesgue measure) on the torus $T=\mathbb{R}^{m}%
\slash\mathbb{Z}^{m}$ and $h\cdot $ is the induced map of $h\in T$ on
differential forms. Thus, taking the infinite Hilbert direct sum of Equation %
\eqref{Hodge q} over $q$, we have 
\begin{equation}
H_{D}^{m}(M)\cong \ker (\Delta _{D}^{(m)})=\ker (D^{(m)})\cap \ker
(D^{\ast (m)}).  \label{Hodge theorem}
\end{equation}%
Note that though each charge component of the above is finite dimensional,
the direct sum may give an infinite number of nonzero components, a priori. Note also that since $\Delta _{\partial }$ and $\Delta _{\overline{\partial }%
}$ preserve $\Omega ^{m}$, so in that case $H_{\overline{\partial }%
}^{m}(M)$ and $H_{\partial }^{m}(M)$ may be subdivided into forms of
type $(r,s,t)$. 
\end{proof}

\section{Relation to Kohn-Rossi cohomology and basic Dolbeault cohomology} \label{Kohn Rossi section}

We now look at how the $\overline\partial$-cohomology can be related to the Kohn-Rossi cohomology \cite%
{KohnRossi1965ExtHolomFcnsToBndryCpxMfd}, \cite{Kohn1964BdriesCpxMflds} which utilizes the differential $\overline{\partial }$
restricted to horizontal forms (called tangential forms in CR manifold
terminology). In this case, as $\xi\lrcorner$ anticommutes with $\overline{\partial }$, the operator $\overline{\partial }$ maps the space of horizontal forms to itself and we have $\overline{\partial }:\Omega
^{r,s,0}\rightarrow \Omega ^{r,s+1,0}$. The same is true for $\partial$. Now, as $d=\partial+\overline\partial+\eta\wedge\mathcal{L}_\xi$, the projection of $d$ to the space of horizontal forms allows to write  ${\rm Proj}\circ d=\partial+\overline\partial$. Therefore, we deduce that $\partial=\partial_{KR}$ and $\overline\partial=\overline\partial_{KR}$, where $\partial_{KR}$ and $\overline\partial_{KR}$ are the Kohn-Rossi differentials. In the same way, we can show that $\partial^\ast=\partial^\ast_{KR}$ and $\overline\partial^\ast=\overline\partial^\ast_{KR}$. As a consequence, the $\overline\partial$-harmonic forms satisfy
\begin{eqnarray*}
    \mathcal{H}_{\overline{\partial}}^{r,s,0}&=&
\mathcal{H}_{\overline{\partial}_{KR}}^{r,s},\\
\mathcal{H}_{\overline{\partial}}^{r,s,1}&=&
\eta\wedge\mathcal{H}_{\overline{\partial}_{KR}}^{r,s}.
\end{eqnarray*}
The same equations hold when restricting to forms of charge $q$, in which case each space is 
finite dimensional. We point out that the operators 
$\Delta_{\overline\partial_C}$ and $\Delta_{\partial_C}$ do not restrict to operators on horizontal forms, because the adjoint $\partial_C^*$ does not anticommute (or commute) with $\xi\lrcorner$ as it is shown in Proposition \ref{del bar C commutator formulas}. 

Using reasoning similar to that in Section \ref{Spectral theory section},
the Kohn-Rossi cohomology $H_{\overline{\partial}_{KR}}^{r,s}\left( M\right) 
$ is the Hilbert sum of the countable set of nontrivial finite dimensional
eigenspaces $H_{\overline{\partial}_{KR}}^{r,s}\left( q\right) $ of $%
\mathcal{L}_{\xi }$ corresponding to eigenvalue $iq $, satisfying%
\begin{equation*}
H_{\overline{\partial}_{KR}}^{r,s}\left( M\right) 
=\bigoplus\nolimits_{q}H_{\overline{\partial}_{KR}}^{r,s}\left( q\right)= \bigoplus\nolimits_{q}H_{\overline{\partial}}^{r,s,0}\left( q\right).
\end{equation*}%
Then, as is the case with ordinary de Rham-Hodge theory, each Kohn-Rossi cohomology class contains a smooth $\Delta _{\overline{%
\partial }}$-harmonic form that minimizes the $L^{2}$-norm within the class.

In the following, we relate the $\partial$-cohomology $\mathcal{H}^\bullet_{\overline\partial}(0)$ to the basic Dolbeault cohomology $\mathcal{H}^\bullet_{\overline\partial}(\mathcal{F})$ associated with the transverse K\"ahler foliation. For this, let $\omega$ be any basic differential form on the foliation $\mathcal{F}$. Since $\xi\lrcorner \omega=0$ and $\mathcal{L}_\xi \omega=0$, we have that $d^{1,1,-1}\omega=0$ and $d^{0,0,1}\omega=0$. Hence, we have that $d\omega=\partial\omega+\overline\partial \omega$. As the foliation is transversally K\"ahlerian, we can write that $d\omega=\partial_b\omega+\overline\partial_b\omega$. Here $\partial_{b}$ and $\overline\partial_{b}$ are the Dolbeault operators on the transverse K\"ahler foliation of the Sasakian manifold. Since both $\partial\omega$ and $\overline\partial\omega$ are basic forms,  we deduce by projecting to $\Omega^{r,s,0}$  that $$\partial\omega=\partial_b\omega,\,\, \overline\partial \omega=\overline\partial_b\omega.$$ In the same way, from \cite{ParkRichardson1996_BasicLaplacian}, we have that $\delta\omega=\delta_b\omega+2\eta\wedge\Lambda\omega$. Again by the fact that $(d^{0,0,1})^\ast\omega=0$ and  $(d^{1,1,-1})^\ast\omega=2\eta\wedge\Lambda\omega$, we deduce that $\delta_b\omega=\partial^\ast\omega+\overline\partial^\ast\omega$. On the other hand,  we have the decomposition $\delta_b=\partial_b^\ast+\overline\partial_b^\ast$ coming from the fact that the foliation is transversally K\"ahlerian. Again by projecting, we deduce that
$$\partial^\ast\omega=\partial_b^\ast\omega,\,\, \overline\partial^\ast \omega=\overline\partial_b^\ast\omega.$$
This implies that the basic Dolbeault cohomology $\mathcal{H}^\bullet_{\overline\partial}(\mathcal{F})$ is a subspace of $\mathcal{H}^\bullet_{\overline\partial}(0)$. Conversely, any form $\omega\in \mathcal{H}^\bullet_{\overline\partial}(0)$ can be decomposed into a horizontal and vertical part as $\omega=\omega_0+\eta\wedge(\xi\lrcorner\omega)$. Since $\mathcal{L}_\xi\omega=0$, this decomposition allows to get that both $\omega_0$ and $\xi\lrcorner\omega$ are basic differential forms in $\mathcal{H}^\bullet_{\overline\partial}(\mathcal{F})$. Hence, we find that 
\begin{equation}\label{eq:decompocohom}
\mathcal{H}^\bullet_{\overline\partial}(0)=\mathcal{H}^\bullet_{\overline\partial}(\mathcal{F})+\eta\wedge \mathcal{H}^\bullet_{\overline\partial}(\mathcal{F}).
\end{equation}

\section{Lefschetz decomposition and its consequences}
\label{Lefschetz conseq section}
In this section, we state the Lefschetz decomposition and study some of its consequences. Recall that a differential form $\omega $ is called \textbf{primitive} if $\Lambda \omega
=0$. As the operators $L,\Lambda$ and $H$ define an $\mathfrak{sl}_2$-representation on the space of all differential forms, we have
\begin{proposition}\label{Lefschetz proposition}
(Lefschetz decomposition on a Sasakian manifold) Any differential form $%
\omega $ of degree $m$ on a Sasakian manifold $(M^{2n+1},g,\xi )$ can be
decomposed as 
\begin{equation*}
\omega =\sum_{0\leq b\leq \left\lfloor \frac{m}{2} \right\rfloor
}L^{b}\omega_P^{m-2b},
\end{equation*}%
where each $\omega_P^{m-2b}$ is primitive  and of degree 
$m-2b$. 
\end{proposition}

Now, we state the following result which will be useful later to get eigenvalue estimates. 
\begin{lemma} \label{commutatorpower}
On the $\left( 2n+1\right) $-dimensional
Sasakian manifold with Lefschetz operator $L$ and adjoint $\Lambda $ and operator $H$, we have for all positive integers $b$, 
\begin{align*}
[L^{b},H] &= 2bL^{b}, \\
[L^{b},\Lambda] &= (-bH-b(b-1))L^{b-1}.
\end{align*}
As a consequence, the map $L^b:\Omega^m\to \Omega^{m+2}$ is injective for $m\leq n-b$ and $\Lambda^b:\Omega^m\to \Omega^{m-2}$, for $m\geq n+2b$.
\end{lemma}
 We establish some results relating the Lefschetz operators  $L$ and $\Lambda$ to
the Laplacian $\Delta _{\overline{\partial }}$.

\begin{lemma}\label{Lambda to the m
commutator with Delta}
For all positive integers $b$, we have $\Delta _{\overline{\partial }}L^{b}=L^{b}\left(
\Delta _{\overline{\partial }}+2bq\right) $ and $\Lambda ^{b}\Delta _{%
\overline{\partial }}=\left( \Delta _{\overline{\partial }}+2bq\right)
\Lambda ^{b}$ on $\Omega ^{\bullet}\left( q\right) $. As a consequence, the maps $L$ and $\Lambda $ on differential forms satisfy%
\begin{eqnarray*}
L^{b} &:&E_{\lambda }\left( \Delta _{\overline{\partial }}\left( q\right)
\right) \rightarrow E_{\lambda +2bq}\left( \Delta _{\overline{\partial }%
}\left( q\right) \right) , \\
\Lambda ^{b} &:&E_{\lambda }\left( \Delta _{\overline{\partial }}\left(
q\right) \right) \rightarrow E_{\lambda -2bq}\left( \Delta _{\overline{%
\partial }}\left( q\right) \right) ,
\end{eqnarray*}%
where $E_{\lambda }\left( \Delta _{\overline{\partial }}\left( q\right)
\right) $ denotes the $\lambda $-eigenspace  of $\Delta _{\overline{%
\partial }}$ on forms of charge $q$.  
\end{lemma}

\begin{proof}
Observe that as operators, we have
\begin{eqnarray*}
\Delta _{\overline{\partial }}L &=&\overline{\partial }\overline{\partial }%
^{\ast }L+\overline{\partial }^{\ast }\overline{\partial }L=\overline{%
\partial }[\overline{\partial }^{\ast },L]+\overline{\partial }L\overline{%
\partial }^{\ast }+\overline{\partial }^{\ast }L\overline{\partial } \\
&=&i\overline{\partial }\partial +L\overline{\partial }\overline{\partial }%
^{\ast }+[\overline{\partial }^{\ast },L]\overline{\partial }+L\overline{%
\partial }^{\ast }\overline{\partial }=i(\overline{\partial }\partial
+\partial \overline{\partial })+L\Delta _{\overline{\partial }} \\
&=&-2i\mathcal{L}_{\xi }L+L\Delta _{\overline{\partial }}=L(\Delta _{%
\overline{\partial }}+2q).
\end{eqnarray*}%
Then%
\begin{equation*}
\Delta _{\overline{\partial }}L^{b}=L(\Delta _{\overline{\partial }%
}+2q)L^{b-1}=L^{2}(\Delta _{\overline{\partial }}+4q)L^{b-2}=...=L^{b}(%
\Delta _{\overline{\partial }}+2bq)
\end{equation*}%
on $\Omega ^{\bullet}\left( q\right) $. The other equation follows by
taking adjoints. The last part can be easily deduced.
\end{proof}

\begin{corollary}
The operator $\Delta _{\overline{\partial }}$ maps primitive forms to
primitive forms, so that $\mathcal{H}_{\overline{\partial },P}^{\bullet}=\ker
\left( \Delta _{\overline{\partial }}:\Omega _{P}^{\bullet}\rightarrow \Omega
_{P}^{\bullet}\right) $ and $\mathcal{H}_{\overline{\partial }%
,P}^{\bullet}\left( q\right) =\ker \left( \Delta _{\overline{\partial }%
}:\Omega _{P}^{\bullet}\left( q\right) \rightarrow \Omega _{P}^{\bullet}\left(
q\right) \right) $.
\end{corollary}

We now use Lemma \ref{Lambda to the m commutator with Delta} to show Theorem \ref{Lefschetz theorem}. We have
\begin{proof}[Proof of Theorem \ref{Lefschetz theorem}]
We consider the Lefschetz decomposition from Proposition \ref{Lefschetz proposition} of a differential form $\omega$ of degree $m$: 
\begin{equation*}
\omega =\sum_{0\leq b\leq \left\lfloor \frac{m}{2} \right\rfloor
}L^{b}\omega_P^{m-2b},
\end{equation*}%
where each $\omega_P^{m-2b}$ is primitive and of degree 
$m-2b$. As $\Lambda$ commutes with $\mathcal{L}_\xi$, we deduce that if $\omega$ is of charge $q$, then is each form $\omega_P^{m-2b}$.
From the lemma above, observe that%
\begin{equation*}
\Delta _{\overline{\partial }}\omega=\sum_{0\leq b\leq \left\lfloor \frac{m}{2} \right\rfloor
}L^{r}\left( \Delta _{\overline{\partial }}+2bq\right) \omega_P^{m-2b},
\end{equation*}%
which gives the Lefschetz decomposition since 
\begin{equation*}
    \Lambda \left( \Delta _{%
\overline{\partial }}+2bq\right) \omega_P^{m-2b}
=\left( \Delta _{%
\overline{\partial }}+2\left( b+1\right) q\right) \Lambda \omega_P^{m-2b}=0.
\end{equation*}
Then $\Delta _{\overline{\partial }}\omega=0$
if and only if $\left( \Delta _{\overline{\partial }}+2bq\right) \omega_P^{m-2b}=0$ for all $b$, yielding $\omega_P^{m-2b}=0$ for $q>0$ and $b>0$. Therefore any element in the cohomology group $\mathcal{H}_{\overline\partial}^\bullet(q)$, for $q>0$, must be primitive.
\end{proof}




We may also consider the dual Lefschetz decomposition, that is, taking advantage of anti-primitive differential form, those forms belonging to the kernel of $L$. For this we utilize the Hodge star operator $*$. Let $\omega$ be any differential form of degree $m$. Then $*\omega$ can be decomposed according to the Lefschetz decomposition:
\[
*\omega=\sum_{0\leq b\leq \left\lfloor \frac{2n+1-m}{2} \right\rfloor }
L^b \omega_P^{2n+1-m-2b},
\]
where each $\omega_P^{\bullet}$ is primitive. Since $*^2=1$ and $*L=\Lambda *$ on differential forms on $M$, we apply $*$ to get
\begin{eqnarray*}
\omega &=&
\sum_{0\leq b\leq \left\lfloor \frac{2n+1-m}{2} \right\rfloor }
\Lambda^b *\omega_{P}^{2n+1-m-2b}\\
&=&\alpha_{P^*}^{m}+\sum_{0\leq b\leq \left\lfloor \frac{2n+1-m}{2} \right\rfloor }
\Lambda^b \alpha_{P^*}^{m+2b},
\end{eqnarray*}
where we have set $\alpha_{P^*}^\bullet=*\omega_P^{2n+1-\bullet}$. The script $P^*$ means that the form is antiprimitive. The following proposition is the dual Lefschetz version of Theorem~\ref{Lefschetz theorem}, which is proved in a similar way. Here,  
$\Delta _{\overline{\partial },P^*}$ is 
the operator $\Delta _{\overline{\partial }}$ 
restricted to antiprimitive forms and 
$\mathcal{H}_{%
\overline{\partial },P^*}^\bullet$
is the intersection of $\overline{\partial }$-harmonic forms with the antiprimitive forms.

\begin{proposition}\label{anti Lefschetz proposition}
For any $q\in\mathbb{R}$,
\begin{equation*}
\mathcal{H}_{\overline{\partial }}^{m}\left( q\right) =\mathcal{H}_{%
\overline{\partial },P^*}^{m}\left( q\right) \oplus
\bigoplus\nolimits_{0\leq b\leq \left\lfloor \frac{2n+1-m}{2} \right\rfloor 
}\Lambda^{b}E_{2bq}\left( \Delta _{\overline{\partial },P^*}\left(
q\right) \right) ,
\end{equation*}%
where $E_{\lambda }\left( \Delta _{\overline{\partial },P^*}(q)\right)$ denotes the eigenspace of $\Delta_{\overline\partial}$ restricted to antiprimitive forms corresponding to eigenvalue $%
\lambda $. In particular, for $q<0$, $\mathcal{H}_{\overline{\partial }%
}^{m}\left( q\right) =\mathcal{H}_{\overline{\partial },P^*}^{m}\left(
q\right) $.
\end{proposition}

\section{Eigenvalue estimates}
\label{eigenvalue estimates section}


In this section, we aim to get several estimates for the eigenvalues of the Hodge Laplacian using Equation \eqref{eq:laplaceDel}. These estimates depend on the charge $q$ when the cohomology $\mathcal{H}^\bullet_{\overline\partial}(q)$ is assumed to be nonzero. Before proving these estimates, we first provide vanishing results for these cohomologies similar to the ones in \cite{Tasker} and \cite{Tanaka1975DiffGeomStudy} with slightly different proofs and assumptions.

For this, we define the conjugate of the Hodge star operator $\ast$ as follows: 
\begin{equation*}
\overline\ast: \Omega^{r,s,t}\to \Omega^{n-r,n-s,1-t}:\,\,\,
\overline\ast\omega=\ast\overline\omega=\overline{*\omega}.
\end{equation*}

\begin{lemma}\label{hodgestar}
On a Sasakian manifold $(M^{2n+1},g,\xi)$, the operator $\overline{*}$ commutes
with $\Delta _{\overline{\partial}}$. Therefore, $%
\mathcal{H}^{r,s,t}_{\overline\partial}(q)$ and $\mathcal{H}%
^{n-r,n-s,1-t}_{\overline\partial}(-q)$ are isomorphic.
\end{lemma}

\begin{proof}
For $\omega \in \Omega ^{\bullet}$, we compute
\begin{eqnarray*}
\overline{\ast }\Delta _{\overline{\partial }}\omega  &=&\overline{\ast }(%
\overline{\partial }\,\,\overline{\partial }^{\ast }+\overline{\partial }%
^{\ast }\,\,\overline{\partial })\omega =\ast (\partial \,\,\partial ^{\ast
}+\partial ^{\ast }\,\,\partial )\overline{\omega } \\
&=&(-1)^{\deg (\omega )}\ast \partial \ast \overline{\partial }\ast 
\overline{\omega }+(-1)^{\deg (\omega )+1}\overline{\partial }\ast
\,\,\partial \overline{\omega } \\
&=&\overline{\partial }^{\ast }\overline{\partial }\ast \overline{\omega }%
+(-1)^{\deg (\omega )+1}\overline{\partial }\ast \,\,\partial \ast ^{2}%
\overline{\omega } \\
&=&\overline{\partial }^{\ast }\overline{\partial }\ast \overline{\omega }+%
\overline{\partial }\,\overline{\partial }^{\ast }\ast \overline{\omega }%
=\Delta _{\overline{\partial }}\ast \overline{\omega }=\Delta _{\overline{%
\partial }}\overline{\ast }\,\omega .
\end{eqnarray*}%
Finally, if $\omega $ is an eigenform of $\mathcal{L}_{\xi }$ which is
associated with the eigenvalue $iq$, then $\ast \overline{\omega }$ is an
eigenform which is associated with the eigenvalue $-iq$. This shows the
isomorphism and finishes the proof.
\end{proof}


\begin{notation}
In all what follows, we use $\langle\cdot,\cdot\rangle$ to indicate the
pointwise Hermitian inner product of sections of $TM\otimes\mathbb{C}$ or of
complex-valued differential forms. We use the notation $\langle\langle\cdot,\cdot\rangle\rangle=%
\int_M\langle\cdot,\cdot\rangle \mathrm{vol}_M$ for the $L^2$-inner product. We use $g(\cdot,\cdot)
$ to mean the pointwise inner product on $TM$ that is complex-linear in both
slots.
\end{notation}

Let us now recall the
Bochner formula: On a Riemannian manifold of dimension $\ell$, the Hodge Laplacian
acting on $m$-forms satisfies \cite[section 1.I]{Besse} 
\begin{equation}\label{bochner}
\Delta = \nabla^*\nabla+\mathcal{B}^{[m]}
\end{equation}
where $\mathcal{B}^{[m]}=\sum_{a,b=1}^{\ell} f_b^\ast\wedge f_a\lrcorner
R(f_a,f_b)$ with $\{f_a\}_{a=1,\ldots,\ell}$ an orthonormal frame of $TM$ and $%
R(X,Y)=[\nabla_X,\nabla_Y]-\nabla_{[X,Y]}$ is the Riemann curvature operator. Recall also the pointwise inequality \cite{GM75} 
\begin{equation}\label{eq:gallomeyerine}
|\nabla\omega|^2\geq \frac 1{m+1}|d\omega|^2+\frac1{\ell-m+1}|\delta \omega|^2,
\end{equation} 
which is valid for any differential form $\omega$. In the following proposition, we give a vanishing result for the  $\overline\partial$-cohomology that depend on some assumptions of the charge $q$ and the Ricci tensor of the manifold. This result covers the one in \cite{Tasker}.

\begin{proposition}\label{vanishing theorem} Let $(M^{2n+1},g,\xi)$ be a compact Sasakian manifold. Assume that ${\rm Ric}\geq 2n$. If $|q+n-r|<\sqrt{n^2+2r}$, then $\mathcal{H}_{\overline\partial}^{r,0,t}(q)=0$, for all $r\in \{0,\ldots,n\}$ and $t\in \{0,1\}$. As a consequence, if $|q|\leq r-1$, then
$\mathcal{H}_{\overline\partial }^{r,0,t}(q)=0$, for all $r>0$ and $t\in \{0,1\}$.
\end{proposition}

\begin{proof}
The proof follows the same lines as in \cite{Tasker} without the additional assumption that the manifold $M$ is Sasaki-Einstein.  Let $\omega\in \mathcal{H}_{\overline\partial}^{r,0,t}(q)$, for $t=0,1$. 
Since $\Lambda$ is an operator of type $(-1,-1,0)$, then $\Lambda\omega=0$.
Hence by \eqref{eq:laplaceDel} and the identity 
$[\Lambda,\overline{\partial}]=-i\partial^*$, we get 
\begin{eqnarray*}
\langle\langle \Delta\omega,\omega\rangle\rangle 
&=&(q^2+2q(n-r-t))\|\omega\|^2+(2q+4(n-r-t+1))\|\xi\lrcorner \omega\|^2.
\end{eqnarray*}
On the other hand, using the Bochner formula \eqref{bochner} and Inequality \eqref{eq:gallomeyerine}, we find
for $t=0$ or ($t=1$ and $r<n$) (resp. for $t=1$, $r=n$), that
\begin{equation}\label{eq:bochnerine}
\langle\langle\Delta\omega,\omega\rangle\rangle\geq \frac{2n-r-t+2}{2n-r-t+1}\langle\langle\mathcal{B}^{[r+t]}\omega,\omega\rangle\rangle.
\end{equation}
(resp. $\langle\langle\Delta\omega,\omega\rangle\rangle\geq \frac{n+2}{n+1}\langle\langle\mathcal{B}^{[n+1]}\omega,\omega\rangle\rangle$). 
Now, the Bochner term can be expressed with respect to an orthonormal frame $\{f_k\}=\{\xi,e_k\}$ of $TM$  (see \cite[section 1.I]{Besse}, \cite[eq. 6]{GrosjeanMinSubmflds2004}) chosen in a way that ${\rm Ric} f_k=\lambda_k f_k$ with $\lambda_k\geq 2n$ (recall that ${\rm Ric}\, \xi=2n\xi$) as 
\begin{eqnarray*}
\langle\mathcal{B}^{[r+t]}\omega,\omega\rangle&=&\sum_{k,l}{\rm Ric}_{k,l}\langle f_k\lrcorner\omega,f_l\lrcorner\omega\rangle+\frac{1}{2}\sum_{a,b,k,l} R_{abkl}\langle f_b\lrcorner f_a\lrcorner\omega,f_l\lrcorner f_k\lrcorner\omega\rangle\\
&\geq &
2n(r+t)|\omega|^2-2(r+t-1)|\xi\lrcorner\omega|^2+\frac{1}{2}\sum_{a,b,k,l}R_{a b k l}\langle e_b\lrcorner e_a\lrcorner\omega,e_l\lrcorner e_k\lrcorner\omega\rangle.
\end{eqnarray*}
Here, we use the symmetry of $R$ and the fact that $R(X,Y)\xi=g(\xi, Y)X-g(\xi, X)Y$ for any $X,Y\in TM$. 
To compute the last sum, we use the expression of the curvature to have $$ R(X^-,Y^-)Z=g(Y^-,Z) X^--g(X^-,Z)Y^-$$ 
for any $X,Y,Z\perp \xi$ which comes from the general formula \cite[Prop. 2.1.1(iii)]{thesisStromenger}
$$R(\phi X,Y)Z+R(X,\phi Y)Z=-g(X,Z) \phi Y+g(\phi Y,Z)X-g(\phi X,Z)Y+g(Y,Z)\phi X.$$
Hence, using that $X^-\lrcorner\omega=0$ since $\omega$ is of type $(r,0,t)$,  the last sum becomes equal to
\begin{eqnarray*} 
\sum_{a,b,k,l}R_{a b k l}\langle e_b\lrcorner e_a\lrcorner\omega,e_l\lrcorner e_k\lrcorner\omega\rangle&=&\sum_{a,b,k,l}g(R(e_a^-,e_b^-)e_k^+,e_l^+)\langle e_b^+\lrcorner e_a^+\lrcorner\omega,e_l^+\lrcorner e_k^+\lrcorner\omega\rangle\\
&=&-2(r+t)(r+t-1)|\omega|^2+4(r+t-1)|\xi\lrcorner\omega|^2.\end{eqnarray*}
Here, we also use the identity $|\omega|^2=\frac{1}{(r+t)(r+t-1)}\sum_{k,l} |f_k\lrcorner f_l\lrcorner\omega|^2$ on the orthonormal frame $\{f_k\}=\{\xi,e_k\}$. Therefore, we find that 
$$\langle\mathcal{B}^{[r+t]}\omega,\omega\rangle\geq (r+t)(2n-r-t+1)|\omega|^2,$$ 
for $t=0,1$. Plugging this last term in \eqref{eq:bochnerine} we find, if $t=0$ or if ($t=1$ and $r<n$), the inequality 
\begin{gather*}
(q^2+2q(n-r-t))\|\omega\|^2+
(2q+4(n-r-t+1))\|\xi\lrcorner\omega\|^2\geq (2n-r-t+2)(r+t)\|\omega\|^2
\end{gather*}
(resp. $r=n$, $t=1$, the inequality $(q^2-2q)\|\omega\|^2+2q\|\xi\lrcorner\omega\|^2\geq (n+2)n\|\omega\|^2$). Now, when $t=0$, we have $\xi\lrcorner \omega=0$, and from the above inequality, we deduce that $\omega=0$ when $|q+n-r|<\sqrt{n^2+2r}$. when $t=1$ and $r<n$, we have that $\|\xi\lrcorner\omega\|^2=\|\omega\|^2$ and, therefore by the same inequality we deduce that $\omega=0$, when $|q+n-r|<\sqrt{(n-1)^2+4r}$ which is less than $\sqrt{n^2+2r}$ for $r<n$. The respective part gives the bound $|q|<\sqrt{n^2+2n}$. This finishes the proof. 
\end{proof}


A Sasakian manifold is said to be positive if its first basic Chern class can be represented by a positive basic definite $(1,1)$-form \cite{BoyerGalickiNakamaye2003PosSasak}. Given such a manifold, there exists a deformed Sasakian structure (Type II deformation) $(\tilde\xi=\xi,\tilde \eta,\tilde\phi,\tilde g)$ such that the deformed metric has strictly positive transverse Ricci curvature \cite{BoyerGalickiNakamaye2003PosSasak}. Recall that this Sasakian structure is given by $\tilde\eta=\eta+d_b^c\varphi$ for some basic function $\varphi$ where $d_b^c=i(\overline\partial_{b}-\partial_{b})$ and $\tilde\phi=\phi-\xi\otimes d_b^c\varphi\circ \phi$, that is $\tilde\phi(X)=\phi(X)-\tilde\eta(\phi(X))\xi$ for all $X\in TM$, and $\tilde g=\tilde\eta\otimes \tilde\eta\oplus d\tilde \eta(\cdot,\phi\cdot)$.  Now by performing a deformation of Type I, that is by taking $\hat\eta=a\tilde\eta,\hat\xi=\frac{1}{a}\tilde \xi$ and $\hat g=a\tilde g+a(a-1)\tilde\eta\otimes\tilde\eta$  for some number $a>0$, we can assume that ${\rm Ric}_{\hat g}\geq 2n$. Indeed, let $k={\rm inf}_M({\rm Ric}^T_{\tilde g})>0$ where ${\rm Ric}^T_{\tilde g}$ is the transversal Ricci curvature of the foliation given by $(\tilde\xi,\tilde g)$. As $\hat g=a\tilde g$ on $\xi^\perp$, we deduce that 
$${\rm Ric}^T_{\hat g}\geq \frac{1}{a}{\rm Ric}^T_{\tilde g}\geq \frac{k}{a}.$$
Hence, we choose $a$ so that $\frac{k}{a}\geq 2n+2$. Now, as $(M,\hat g,\hat\xi)$ is Sasakian, we have that 
$${\rm Ric}_{\hat g}|_{\hat\xi^\perp}={\rm Ric}^T_{\hat g}-2{\rm Id}_{\hat\xi^\perp} \geq \frac{k}{a}-2\geq 2n.$$
Since ${\rm Ric}_{\hat g}\hat\xi=2n\hat\xi$, we finally get ${\rm Ric}_{\hat g}\geq 2n$. A direct consequence of Proposition \ref{vanishing theorem} and the decomposition \eqref{eq:decompocohom}, we deduce the following result that was proven in \cite{Nozawa2014DefSasakiMetr}.

\begin{corollary}  Let $(M^{2n+1},g,\xi )$ be a compact positive Sasakian  manifold. Then, we get   $$\mathcal{H}_{\overline\partial }^{r,0}(\mathcal{F})=0,$$ 
for all $r>0$.
\end{corollary}

Now, we prove an eigenvalue estimate for the Dolbeault Laplacian $\Delta_{\overline\partial}$ and characterize its limiting case. We have

\begin{proof}[Proof of Theorem \ref{thm:eigenestimatedeltabar}] Let $m \in\{0,\ldots,n-1\}$ and take a nonzero eigenform $\omega$ of degree $m$ such that $\Delta_{\overline\partial}\omega=\lambda\omega$ with $\mathcal{L}_\xi\omega=iq\omega$, for some $q\leq 0$. Choose a positive number $b$ such that $b\leq n-m$. The form $\alpha_b:=L^b\omega$ does not vanish from the injectivity of $L^b$ in Lemma \ref{commutatorpower}. Now from Lemma \ref{Lambda to the m commutator with Delta}, we have that $\Delta_{\overline\partial}\alpha_b=(\lambda+2bq)\alpha_b$. Therefore, we find 
\begin{equation}\label{eq:inelambdaq}
\lambda\geq -2bq.
\end{equation}
On the other hand, applying Equality \eqref{laplaceDels} to $\alpha_b$ yields the identity 
\begin{equation}\label{eq:deltaalphabestim}
\Delta_{\partial}\alpha_b=(\lambda+2q(n-m-b))\alpha_b+2q\eta\wedge\xi\lrcorner\alpha_b.
\end{equation}
Here, two cases occur. If $\xi\lrcorner\omega=0$, then $\xi\lrcorner\alpha_b=0$ and we get that $\lambda\geq -2q(n-m-b)$. If $\xi\lrcorner\omega\neq 0$, then $\xi\lrcorner\alpha_b$ is also a nonzero $(m-1)$-form again from the injectivity of $L^b$, as $m-1+b<n$. Taking the interior product of \eqref{eq:deltaalphabestim} with $\xi$ yields $\Delta_{\partial}(\xi\lrcorner\alpha_b)=(\lambda+2q(n-m-b+1))\xi\lrcorner \alpha_b$. This gives that 
$$\lambda\geq -2q(n-m-b+1)\geq -2q(n-m-b),$$
as $q\leq 0$. Hence, in both cases we get that $\lambda\geq-2q(n-m-b)$. Finally adding this last inequality with \eqref{eq:inelambdaq} yields the result. In the following, we discuss the equality case of this estimate. If $\lambda=-q(n-m)$, then we have equalities in all the above inequalities. Hence, for  $q=0$ the eigenvalue $\lambda$ is necessarily zero. Now, if $q<0$, the form $\omega$ should then be horizontal and from \eqref{eq:inelambdaq}, we get that $n-m=2b$. Moreover, we have that $\Delta_{\overline\partial}\alpha_b=0$. Therefore, $\alpha_b\in \mathcal{H}^{m+2b}_{\overline\partial}(q)$, for $q<0$. Proposition \ref{anti Lefschetz proposition} tells us that $L\alpha_b=0$, which means that $L^{b+1}\omega=0$. Now we need to check the injectivity of $L^{b+1}$ on $m$-forms. In fact, when $n=1$, then $m=0$ and the condition $2b=n-m$ cannot be satisfied. The same is true when $n\geq 2$ and $m=n-1$. Thus, the equality cannot be attained in these two cases. Hence, we can take $n\geq 2$ and $m\leq n-2$. In this case, the map $L^{b+1}$ is injective, since $m+b+1=\frac{m+n+2}{2}\leq n$. Therefore the form $\omega$ must vanish which gives a contradiction. We deduce that the equality can be realized if and only if $\lambda=q=0$. This finishes the proof.
\end{proof}
\begin{remark}
Notice that the assumption of $q\leq 0$ can be dropped if we consider a horizontal eigenform $\omega$, as for example in the case of functions.
\end{remark}

In the following, we give eigenvalue estimates for the Hodge Laplacian that mainly uses Equation \eqref{eq:laplaceDel}. For all $m \in\{0,\ldots,2n+1\},$ we denote by $\lambda'_{1,m}$ (resp. $\lambda''_{1,m}$) the first positive eigenvalue of the Hodge Laplacian restricted to closed forms (resp. co-closed forms) of degree $m$. Recall that $\lambda''_{1,m}=\lambda'_{1,m+1}$ and by Hodge theorem, we have the equality 
$$\lambda_{1,m}={\rm min} (\lambda_{1,m}', \lambda_{1,m}'').$$

\begin{proof}[Proof of Theorem \ref{eigenvalue estimate 1}] Let $\omega$ be an element in $\mathcal{H}_{\overline\partial,P}^{m}(0)$. Then, from \eqref{laplaceDels}, we get that $\Delta_\partial\omega=0$ which gives that $\partial\omega=0$ and $\partial^\ast\omega=0$. Therefore Equation \eqref{eq:laplaceDel} allows to write that 
$$
\Delta\omega=4\eta\wedge\xi\lrcorner H\omega.
$$
Hence 
\begin{eqnarray*}
\langle\langle\Delta\omega,\omega\rangle\rangle&=&4\langle\langle H\xi\lrcorner \omega,\xi\lrcorner\omega\rangle\rangle =4(n-m+1)||\xi\lrcorner\omega||^2.
\end{eqnarray*}
If $m>n+1$, the r.h.s of the above equality becomes nonpositive while the l.h.s. is nonnegative. Hence, $\Delta\omega=0$ and $\xi\lrcorner\omega=0$. In this case, the form $\omega$ must satisfy $\eta\wedge\omega=0$ from \cite{Tachibana1965OnHarmonicTensors}, and therefore, it vanishes. When $m=n+1$, we have that $\Delta\omega=0$ and, thus, $\omega$ is an element in $H^{n+1}(M)$. 

Assume now that $\mathcal{H}_{\overline\partial,P}^{m}(0)\neq 0$ for some $m<n$. We consider two cases: the first case is when $\xi\lrcorner\omega=0$. Then the nonzero form $\eta\wedge\omega\in \mathcal{H}^{m+1}_{\overline\partial, P}(0)$ and we have 
\begin{eqnarray*}
\langle\langle\Delta(\eta\wedge\omega),\eta\wedge\omega\rangle\rangle&=& 4(n-m)||\xi\lrcorner(\eta\wedge\omega)||^2=4(n-m)||\eta\wedge\omega||^2.
\end{eqnarray*}
Now, the form $\eta\wedge\omega$ is co-closed. Indeed, as $\Delta\omega=0$, we deduce that $\delta\omega=0$ and from the rule $\{\delta,\eta\wedge\}=-\mathcal{L}_\xi$, we get that $\delta(\eta\wedge\omega)=0$. We also have that $\eta\wedge\omega$ is orthogonal to any harmonic form $\omega_0$ of degree $m+1<n+1$, as $\langle\eta\wedge\omega,\omega_0\rangle=\langle\omega,\xi\lrcorner\omega_0\rangle=0$ by the fact that $\omega_0$ is horizontal \cite{Tachibana1965OnHarmonicTensors}. We then deduce the estimate for the first positive eigenvalue $\lambda_{1,m+1}''$ 
$$\lambda_{1,m+1}''\leq 4(n-m).$$
The second case is when $\phi:=\xi\lrcorner\omega$ is not zero. Therefore, $\xi\lrcorner\phi=0$ and $\phi$ is in $\mathcal{H}_{\overline\partial,P}^{m-1}(0)$. Hence, from the first case we find that 
$$\lambda_{1,m}''\leq 4(n-m+1).$$
As $\lambda''_{1,m}=\lambda'_{1,m+1}$, we deduce the estimate 
$$\lambda_{1,m+1}<4(n-m+1).$$
To prove the last part, we take any positive integer $b\leq \frac{n-m}{2}$ and consider the $(m+2b)$-differential form  given by $\alpha_b:=L^b\omega$. From Lemma \ref{commutatorpower}, the form $\alpha_b$ does not vanish. As $\omega\in \mathcal{H}_{\overline\partial,P}^{m}(0)$, Lemma \ref{Lambda to the m
commutator with Delta} tells us that $\Delta_{\overline\partial}\alpha_b=0$, which also gives that $\Delta_{\partial}\alpha_b=0$ from Equation \eqref{laplaceDels}. Hence, we get $\overline\partial^\ast \alpha_b=0$ and $\partial^\ast \alpha_b=0$. Applying Equation \eqref{eq:laplaceDel} to the form $\alpha_b$ yields after using that, $\overline\partial\alpha_b=0$ and 
\begin{eqnarray}\label{eq:llambdaalphab}
L\Lambda\alpha_b&=&bLHL^{b-1}\omega+b(b-1)\alpha_b\nonumber\\
&=&bH\alpha_b+b(b+1)\alpha_b\nonumber\\
&=&b(n-m-b+1)\alpha_b+b\eta\wedge\xi\lrcorner\alpha_b
\end{eqnarray}
from Lemma \ref{commutatorpower}, the following
$$\Delta\alpha_b=4b(n-m-b+1)\alpha_b+4(n-m-b+1)\eta\wedge\xi\lrcorner\alpha_b.$$
Now, we proceed as before. We first consider the case when $\xi\lrcorner\omega=0$. In this case, we have that $\xi\lrcorner\alpha_b=0$ and $\alpha_b$ is orthogonal to any harmonic form $\omega_0$ of degree $m+2b<n+1$, as $\langle\alpha_b,\omega_0\rangle=\langle\omega,\Lambda^b\omega_0\rangle=0$ using that $\omega_0$ is primitive by \cite{Tachibana1965OnHarmonicTensors}. We deduce 
$$\lambda_{1,m+2b}'\leq 4b(n-m-b+1).$$
Notice here that the form $\alpha_b$ is closed, since $\partial\alpha_b=\overline\partial\alpha_b=0$ and $d^{1,1,-1}\alpha_b=d^{0,0,1}\alpha_b=0$, as $\alpha_b$ is horizontal and $\mathcal{L}_\xi\alpha_b=0$. If $\xi\lrcorner\omega$ is not zero, then  $\phi:=\xi\lrcorner\omega\in \mathcal{H}_{\overline\partial,P}^{m-1}(0)$. We replace $b$ by $b+1$ and $m$ by $m-1$ in the previous step to get that
$$\lambda_{1,m+2b+1}'\leq 4(b+1)(n-m-b+1).$$
We point out that the form $L^{b+1}\phi$ is orthogonal to any harmonic form of degree $m-1+2(b+1)$ as the condition $m-1+2(b+1)\leq n+1$ is still satisfied. Since $\lambda_{1,m+2b+1}'=\lambda_{1,m+2b}''$, we deduce that 
$$\lambda_{1,m+2b}< 4(b+1)(n-m-b+1).$$
This finishes the proof.
\end{proof}

\begin{remark} It is not difficult to check that the de Rham cohomology group $\mathcal{H}^m(M)$ can be written as $$\mathcal{H}^m(M)=\oplus_{r+s=m} \mathcal{H}_{\overline\partial,P}^{r,s,0}(0) \quad\text{for $m<n+1$, and}\quad \mathcal{H}^{n+1}(M)=\oplus_{r+s=n} \mathcal{H}_{\overline\partial,P}^{r,s,1}(0).$$ Indeed, take any harmonic form $\omega$ of degree $m<n+1$, then by \cite{Tachibana1965OnHarmonicTensors}, the form $\omega$ is horizontal and primitive, that is $\xi\lrcorner\omega=0$ and $\Lambda\omega=0$. As $\mathcal{L}_\xi\omega=0$, Equation \eqref{eq:laplaceDel} applied to $\omega$ yields $\Delta_{\overline\partial} \omega=0$. Thus, $\omega\in \mathcal{H}_{\overline\partial,P}^{m}(0)$ and the decomposition follows. When $m=n+1$, again by Tachibana \cite{Tachibana1965OnHarmonicTensors},  the form $
\omega$ must be primitive and $\eta\wedge\omega=0$. Therefore, from the relation $\{\delta,\xi
\lrcorner\}=2\Lambda$, we get $\xi\lrcorner\omega\in\mathcal{H}^{n}(M)$. Hence $\mathcal{H}^{n+1}(M)\subset
\eta\wedge \mathcal{H}^n(M)$. For the converse, let $\omega\in \mathcal{H}
^n(M)$. Hence, $L\omega=0$ by \cite{Tachibana1965OnHarmonicTensors}, and
therefore $d(\eta\wedge\omega)=2L\omega=0$ and $\delta(\eta\wedge\omega)=-\mathcal{L}_\xi\omega=0$. Hence $\eta\wedge\omega\in \mathcal{H}^{n+1}(M)$. Thus, we get that $\mathcal{H}^{n+1}(M)=\eta\wedge \mathcal{H}^n(M)$. 
\end{remark}
In the following, we give an eigenvalue estimate in case when the charge $q$ is not zero. We have

\begin{proof}[Proof of Theorem \ref{eigenvalue estimate 2}]
Let $\omega$ be an element in $\mathcal{H}_{\overline\partial}^{m}(q)$. Using %
\eqref{laplaceDels}, we have 
\begin{eqnarray*}
0&\leq& \langle\langle \Delta _{\partial}\omega,\omega\rangle\rangle
=2q\langle\langle H\omega,\omega\rangle\rangle  \notag \\
&=&2q(n-m)\|\omega\|^2 +2q\| \xi\lrcorner \omega\|^2\\
&\leq& 0,
\end{eqnarray*}
if $q<0$ and $m<n$. This shows the first part. The respective part comes from the fact that the Hodge star operator commutes with $\Delta_{\overline\partial}$ from Proposition \ref{hodgestar}.

When $m=n$ and $q<0$, we get from the previous inequality that $\xi\lrcorner \omega=0$. Therefore $H\omega=0$. Using Proposition \ref{anti Lefschetz proposition}, we have that $L\omega=0$ and, thus, $L\Lambda\omega=0$. By taking the Hermitian product with $\omega$, we deduce $\Lambda\omega=0$. Since $[\Lambda,\overline\partial]=-i\partial^\ast $, this yields $\partial^\ast \omega =0$ and 
\begin{equation*}
\delta\omega =\partial^\ast \omega +\overline{\partial }^\ast\omega +2\eta\wedge \Lambda \omega-\xi\lrcorner \mathcal{L}_{\xi }\omega =0.
\end{equation*}%
Equation \eqref{eq:laplaceDel} gives $\Delta \omega =q^{2}\omega $.  The respective part is also a consequence of Proposition \ref{hodgestar}.

When $q>0$ and $m<n+1$, we recall from Theorem~\ref{Lefschetz theorem} that $\Lambda\omega=0$. Thus, the commutator $[\Lambda,\overline\partial]=i\partial^\ast$ allows to deduce that $\partial^\ast\omega=0$. For any nonnegative number $b$ with $b\leq n-m$, take the form $\alpha_b:=L^b\omega$. By  Lemma \ref{commutatorpower}, the form $\alpha_b$ does not vanish. Now, we have that $\overline\partial\alpha_b=0$ and, from the commutator  $[\partial^\ast,L]=-i\overline\partial$, we get that $\partial^\ast \alpha_b=0$. Also, one can easily prove by induction from $[\overline\partial^\ast,L]=i\partial$ that $\overline\partial^\ast \alpha_b=ib \partial \alpha_{b-1}$. As a consequence of Lemma \ref{Lambda to the m
commutator with Delta}, we find 
\begin{equation}\label{eq:laplacianalpha}
\Delta_{\overline\partial}\alpha_b=2bq\alpha_b.
\end{equation}

 Applying  Equation \eqref{eq:laplaceDel} to $\alpha_b$ yields 
\begin{eqnarray*}
\Delta\alpha_b&=&2\Delta_{\overline\partial}\alpha_b+2qH\alpha_b+2i\eta\wedge\overline\partial^\ast\alpha_b-2i\xi\lrcorner\partial\alpha_b+4L\Lambda\alpha_b+4\eta\wedge\xi\lrcorner H\alpha_b+q^2\alpha_b\\
&\stackrel{\eqref{eq:laplacianalpha}, \eqref{eq:llambdaalphab}}{=}&\left(2q(n-m)+4b(n-m-b+1)+q^2\right)\alpha_b+2(q+2(n-m-b+1))\eta\wedge\xi\lrcorner\alpha_b\\&&+2i\eta\wedge\overline\partial^\ast\alpha_b-2i\xi\lrcorner\partial\alpha_b. 
\end{eqnarray*}
Two cases occur. The first case is when $\xi\lrcorner \omega=0$. Here, we get $\xi\lrcorner\alpha_b=0$ and, thus,
\begin{eqnarray*}
\langle\langle\Delta\alpha_b,\alpha_b\rangle\rangle
=(2q(n-m)+4b(n-m-b+1)+q^2)||\alpha_b||^2.
\end{eqnarray*}
The form $\alpha_b$ is $L^2$-orthogonal to any harmonic form $\beta$ since 
$$iq\langle\langle\alpha_b,\beta\rangle\rangle=\langle\langle\mathcal{L}_\xi\alpha_b,\beta\rangle\rangle=-\langle\langle\alpha_b,\mathcal{L}_\xi\beta\rangle\rangle=0$$
with $q\neq 0$, we deduce by the min-max principle that the first positive eigenvalue $\lambda_{1,m+2b}$ satisfies
$$\lambda_{1,m+2b}\leq q^2+2q(n-m)+4b(n-m-b+1).$$
However, when $b=0$, the form $\alpha_0=\omega$ is co-closed, as 
$$\delta\omega=\partial^\ast\omega+\overline\partial^\ast\omega-\xi\lrcorner\mathcal{L}_\xi\omega+2\eta\wedge \Lambda\omega=0.$$
Hence the above inequality reduces to 
$$\lambda_{1,m}''\leq q^2+2q(n-m).$$
The second case is when $\xi\lrcorner\omega\neq 0$. In this case, the form $\phi=\xi\lrcorner\omega\in \mathcal{H}^{m-1}_{\overline\partial}(q)$. Therefore, we are back to the previous case by taking the form $L^{b+1}\phi$ and we get, for $b=0$, 
$$\lambda_{1,m-1}''\leq q^2+2q(n-m+1),$$ 
and, for $b\geq 1$, that 
$$\lambda_{1,m+2b+1}\leq q^2+2q(n-m+1)+4(b+1)(n-m-b+1).$$
Combining the two cases, we deduce after using $\lambda_{1,m-1}''=\lambda_{1,m}'$ that 
$$\lambda_{1,m}<q^2+2q(n-m+1),$$
and, for $1\leq b\leq n-m$, 
$${\rm min}(\lambda_{1,m+2b}, \lambda_{1,m+2b+1})<q^2+2q(n-m+1)+4(b+1)(n-m-b+1).$$
\end{proof}
When the Hodge Laplacian acts on the space of horizontal or vertical forms followed by a projection, a lower bound estimate can be deduced from Equation \eqref{eq:laplaceDel} that also depends on the charge $q$. We have
\begin{theorem} Let $(M^{2n+1},g,\xi)$ be a compact Sasakian manifold. If $P_v=\eta\wedge\xi\lrcorner$ and $P_h=1-P_v$ are the pointwise projections to vertical and horizontal forms, respectively, then the spectrum of 
\[
P_h\Delta P_h + P_v \Delta P_v 
\]
on differential $m$-forms of charge $q$
is bounded from below by $${\rm min}(q^{2}+2q(n-m),q^2+2(q+2)(n-m+1)).$$

\end{theorem}
\begin{proof}Take any differential form $\omega$ of degree $m$ with $\mathcal{L}_\xi\omega =iq\omega$. Then again from Equation \eqref{eq:laplaceDel} we have
\begin{eqnarray*}
\langle\langle \Delta \omega ,\omega \rangle\rangle &=&2\langle\langle \Delta _{%
\overline{\partial }}\omega ,\omega \rangle\rangle  +2q\langle\langle H\omega ,\omega \rangle\rangle
+2\langle\langle i  (\eta \wedge ( \overline{\partial }%
^{\ast }-\partial ^{\ast } ) +\xi \lrcorner ( \overline{\partial }%
-\partial  )  ) \omega ,\omega \rangle\rangle  \\
&& 
+4||\Lambda\omega||^2
+4\langle\langle \xi\lrcorner H\omega
,\xi\lrcorner\omega \rangle\rangle 
+q^{2}||\omega||^2 
\\
&\geq& 2q(n-m)||\omega||^2+2(q+2(n-m+1))||\xi\lrcorner\omega||^2+q^{2}||\omega||^2\\&&
 +2\langle\langle i  (\eta \wedge ( \overline{\partial }%
^{\ast }-\partial ^{\ast } ) +\xi \lrcorner ( \overline{\partial }%
-\partial  )  ) \omega ,\omega \rangle\rangle.
\end{eqnarray*}
When $\omega$ is horizontal or vertical, the last term clearly vanishes. By writing $\langle\langle (P_h\Delta P_h + P_v \Delta P_v)\omega,\omega\rangle\rangle=\langle\langle \Delta P_h\omega,P_h\omega\rangle\rangle + \langle\langle \Delta P_v\omega,P_v\omega\rangle\rangle$, we deduce the statement of the theorem. 
\end{proof}

We end this section by describing a part of the spectrum of the Hodge Laplacian in terms of the charge $q$ and eigenvalues of the scalar Dolbeault Laplacian. 

\begin{proof}[Proof of Theorem \ref{eigenvalue}] Let $f$ be any function such that $\Delta_{\overline \partial} f=\lambda f$ with $\xi(f)=iq f$. We consider the differential $2m$-form $\alpha_m:=f L^m(1)=L^m(f)$ for any positive integer $m$. Using Lemma \ref{Lambda to the m
commutator with Delta}, we have that 
\begin{equation} 
\Delta_{\overline\partial}\alpha_m=(\lambda+2mq)\alpha_m.\label{eq:laplaciaalphambar}
\end{equation} 
and, thus, 
\begin{equation}\label{eq:deltaalpham}
\Delta_{\partial}\alpha_m=(\lambda+2q(n-m))\alpha_m,
\end{equation}
by Equation \eqref{laplaceDels}.  From Lemma \ref{commutatorpower}, we have that 
$\Lambda L^{m}(1)=m(n-m+1) L^{m-1}(1)$ 
and
 \begin{equation}\label{eq:partiallambda}
L\Lambda\alpha_m=fL\Lambda L^m(1)=m(n-m+1)\alpha_m.
\end{equation}
Hence, Equations \eqref{eq:laplaceDel}, \eqref{eq:laplaciaalphambar} and \eqref{eq:partiallambda} imply
\begin{equation*}
\Delta\alpha_m=(2\lambda+2qn+q^2+4m(n-m+1))\alpha_m+2i\eta \wedge
\left( \overline{\partial }^{\ast }\alpha_m-\partial ^{\ast }\alpha_m\right).
\end{equation*}
Now, we are going to compute the Laplacian of the forms $\eta \wedge \overline{\partial }^{\ast }\alpha_m$ and  $\eta\wedge {\partial }^{\ast }\alpha_m$. First,  we can easily show 
\begin{equation}\label{eq:partialstarpartialalphe}
\overline\partial^\ast\alpha_m=im\partial\alpha_{m-1}\quad\text{and}\quad\partial^\ast\alpha_m=-im\overline\partial\alpha_{m-1}.
\end{equation}
Here, we use the commutators from Proposition \ref{Kaehler Identity prop} and $\partial,\overline\partial$ commute with $L$. Again using Equation \eqref{eq:laplaceDel} applied to the vertical form $\eta \wedge \overline{\partial }^{\ast }\alpha_m$ gives
\begin{eqnarray*}
\Delta(\eta \wedge \overline{\partial }^{\ast }\alpha_m)&=&2\Delta _{\overline{\partial }}(\eta \wedge \overline{\partial }^{\ast }\alpha_m)+\left(q^2+2q(n-2m+1)+4(n-2m+1)\right) \eta \wedge \overline{\partial }^{\ast }\alpha_m\\&&-2i\overline\partial\overline\partial^\ast\alpha_m+2i \partial\overline\partial^\ast\alpha_m+4\eta \wedge L\Lambda  \overline{\partial }^{\ast }\alpha_m\\
&\stackrel{\eqref{eq:partialstarpartialalphe},\eqref{eq:laplaciaalphambar}}{=}&\left(2\lambda+q^2+2q(n+1)+4(n-2m+1)\right )\eta \wedge \overline{\partial }^{\ast }\alpha_m+2m\overline\partial\partial \alpha_{m-1}\\&&+4\eta \wedge (\overline{\partial }^{\ast }L-i\partial)\Lambda  \alpha_m\\
&\stackrel{\eqref{eq:partiallambda}}{=}&\left(2\lambda+q^2+2q(n+1)-4m(m-n)\right )\eta \wedge \overline{\partial }^{\ast }\alpha_m+2m\overline\partial\partial \alpha_{m-1}.
\end{eqnarray*}
In the last equality, we also use that
\begin{equation*}
i\partial\Lambda\alpha_m=i\partial(f\Lambda L^{m}(1))=im(n-m+1)\partial(fL^{m-1}(1))=(n-m+1)\overline\partial^\ast\alpha_m,
\end{equation*}
from Equality \eqref{eq:partialstarpartialalphe}. In the same way and by a similar computation as above using \eqref{eq:laplaceDel} applied to the vertical form $\eta \wedge {\partial }^{\ast }\alpha_m$, we have  
\begin{eqnarray*}
\Delta(\eta \wedge {\partial }^{\ast }\alpha_m)&=&2\Delta _{\overline{\partial }}(\eta \wedge {\partial }^{\ast }\alpha_m)+\left(q^2+2q(n-2m+1)+4(n-2m+1)\right) \eta \wedge {\partial }^{\ast }\alpha_m\\&&-2i\overline\partial\partial^\ast\alpha_m+2i \partial\partial^\ast\alpha_m+4\eta \wedge L\Lambda  {\partial }^{\ast }\alpha_m\\
&\stackrel{\eqref{laplaceDels},\eqref{eq:partialstarpartialalphe}}{=}&2\eta\wedge (\Delta_{\partial}-2qH)\partial^\ast\alpha_m+\left(q^2+2q(n-2m+1)+4(n-2m+1)\right) \eta \wedge {\partial }^{\ast }\alpha_m\\&&+2m \partial\overline\partial\alpha_{m-1}+4\eta \wedge (\partial^\ast L+i\overline\partial) \Lambda\alpha_m\\
&\stackrel{\eqref{eq:deltaalpham}}{=}&\left(2\lambda+q^2+2q(n-1)-4m(m-n)\right )\eta \wedge {\partial }^{\ast }\alpha_m+2m(-\overline\partial\partial-2iqL)  \alpha_{m-1}\\
&\stackrel{\eqref{eq:partialoverpartial}}{=}&\left(2\lambda+q^2+2q(n-1)-4m(m-n)\right )\eta \wedge {\partial }^{\ast }\alpha_m-2m \overline\partial\partial \alpha_{m-1}-4imq\alpha_m. 
\end{eqnarray*}
In the above computation, we use $i\overline\partial\Lambda\alpha_m=-(n-m+1)\partial^\ast\alpha_m$, which can be shown in the same way as before. We still need to compute $\Delta(\overline\partial\partial\alpha_{m-1})$. First, we have 
\begin{eqnarray}\label{eq:llambdapartialalpham}
L\Lambda(\overline\partial\partial\alpha_{m-1})&=&L(\overline\partial\Lambda-i\partial^\ast)\partial\alpha_{m-1}\nonumber\\
&=&L\overline\partial(\partial\Lambda+i\overline\partial^\ast)\alpha_{m-1}-iL\partial^\ast\partial\alpha_{m-1}\nonumber\\
&\stackrel{\eqref{eq:partiallambda}}{=}&(m-1)(n-m+2)\overline\partial\partial\alpha_{m-1}+i\overline\partial L\overline\partial^\ast\alpha_{m-1}-iL\partial^\ast\partial\alpha_{m-1}\nonumber\\
&=&(m-1)(n-m+2)\overline\partial\partial\alpha_{m-1}+i\overline\partial (\overline\partial^\ast L-i\partial)\alpha_{m-1}-i(\partial^\ast L+i\overline\partial)\partial\alpha_{m-1}\nonumber\\
&=&(mn-m^2+3m-n)\overline\partial\partial\alpha_{m-1}+i\overline\partial\overline\partial^\ast\alpha_m-i\partial^\ast\partial\alpha_m\nonumber\\
&\stackrel{\eqref{eq:partialstarpartialalphe}}{=}&(mn-m^2+2m-n)\overline\partial\partial\alpha_{m-1}-i(\Delta_\partial-\partial\partial^\ast)\alpha_m\nonumber\\
&\stackrel{\eqref{eq:deltaalpham},\eqref{eq:partialstarpartialalphe}}{=}&(mn-m^2+2m-n)\overline\partial\partial\alpha_{m-1}-i(\lambda+2q(n-m))\alpha_m+m\partial\overline\partial\alpha_{m-1}\nonumber\\
&\stackrel{\eqref{eq:partialoverpartial}}{=}&(n-m)(m-1)\overline\partial\partial\alpha_{m-1}-i(\lambda+2qn)\alpha_m.
\end{eqnarray}
We also compute 
\begin{eqnarray}\label{eq:partialstarminuspartial}
(\overline\partial^\ast-\partial^\ast)(\overline\partial\partial\alpha_{m-1})&\stackrel{\eqref{eq:partialstarpartialcom}}{=}&(\Delta_{\overline\partial}-\overline\partial\overline\partial^\ast)\partial\alpha_{m-1}+\overline\partial\partial^\ast\partial\alpha_{m-1}\nonumber\\
&\stackrel{\eqref{laplaceDels},\eqref{eq:partialstarpartialalphe}}{=}&(\Delta_{\partial}-2qH)\partial\alpha_{m-1}+\overline\partial(\Delta_\partial-\partial\partial^\ast)\alpha_{m-1}\nonumber\\
&\stackrel{\eqref{eq:deltaalpham},\eqref{eq:partialoverpartial}}{=}&(\lambda+2mq)\partial\alpha_{m-1}+(\lambda+2q(n-m+1))\overline\partial\alpha_{m-1}\nonumber\\&&-(-\partial\overline\partial-2iqL)\partial^\ast\alpha_{m-1}\nonumber\\
    &=&(\lambda+2mq)\partial\alpha_{m-1}+(\lambda+2q(n-m+1))\overline\partial\alpha_{m-1}\nonumber\\&&+2iq(\partial^\ast L+i\overline\partial)\alpha_{m-1}\nonumber\\
    &\stackrel{\eqref{eq:partialstarpartialalphe}}{=}&\frac{-i(\lambda+2mq)}{m}\overline\partial^\ast\alpha_{m}+\frac{i(\lambda+2qn)}{m}\partial^\ast\alpha_{m}.
\end{eqnarray}
Equation \eqref{eq:laplaceDel} applied to $\overline\partial\partial\alpha_{m-1}$ gives after using \eqref{eq:llambdapartialalpham} and \eqref{eq:partialstarminuspartial} the following:
\begin{eqnarray*}
\Delta (\overline\partial\partial\alpha_{m-1})&=&2\Delta _{\overline{\partial }}(\overline\partial\partial\alpha_{m-1})+(2q(n-2m)+q^2)\overline\partial\partial\alpha_{m-1}+
\frac{2(\lambda+2mq)}{m}\eta \wedge\overline\partial^\ast\alpha_{m}\\&&-\frac{2(\lambda+2qn)}{m}\eta\wedge\partial^\ast\alpha_{m} +4(n-m)(m-1)\overline\partial\partial\alpha_{m-1}-4i(\lambda+2qn)\alpha_m\\
&\stackrel{\eqref{laplaceDels}}{=}&2\overline\partial(\Delta_\partial-2qH)\partial\alpha_{m-1}+\left(2q(n-2m)+q^2+4(n-m)(m-1)\right)\overline\partial\partial\alpha_{m-1}\\&&+
\frac{2(\lambda+2mq)}{m}\eta \wedge\overline\partial^\ast\alpha_{m}-\frac{2(\lambda+2qn)}{m}\eta\wedge\partial^\ast\alpha_{m} -4i(\lambda+2qn)\alpha_m\\
&\stackrel{\eqref{eq:deltaalpham}}{=}&\left(2\lambda+2nq+q^2+4(n-m)(m-1)\right)\overline\partial\partial\alpha_{m-1}+
\frac{2(\lambda+2mq)}{m}\eta \wedge\overline\partial^\ast\alpha_{m}\\&&-\frac{2(\lambda+2qn)}{m}\eta\wedge\partial^\ast\alpha_{m} -4i(\lambda+2qn)\alpha_m.
\end{eqnarray*}
We summarize the previous computations and write the equations as a linear system:
\begin{equation}\label{eq:system}
\left\{
\begin{matrix}
\Delta\alpha_m=A\alpha_m+2i\eta \wedge
\overline{\partial }^{\ast }\alpha_m-2i\eta\wedge \partial ^{\ast }\alpha_m\\\\
\Delta(\eta \wedge \overline{\partial }^{\ast }\alpha_m)=B\eta \wedge \overline{\partial }^{\ast }\alpha_m+2m\overline\partial\partial \alpha_{m-1}\\\\
\Delta(\eta \wedge {\partial }^{\ast }\alpha_m)=C\eta \wedge {\partial }^{\ast }\alpha_m-2m \overline\partial\partial \alpha_{m-1}-4imq\alpha_m\\\\
\Delta (\overline\partial\partial\alpha_{m-1})=D\overline\partial\partial\alpha_{m-1}+
\frac{2(\lambda+2mq)}{m}\eta \wedge\overline\partial^\ast\alpha_{m}-\frac{2(\lambda+2qn)}{m}\eta\wedge\partial^\ast\alpha_{m}-4i(\lambda+2qn)\alpha_m
\end{matrix}\right.
\end{equation}
where the numbers $A, B, C,$ and $D$ are given by $$A=2\lambda+2qn+q^2+4m(n-m+1),\,\,\, B=2\lambda+q^2+2q(n+1)-4m(m-n)$$ and $$C=2\lambda+q^2+2q(n-1)-4m(m-n),\,\,\, D=2\lambda+2nq+q^2+4(n-m)(m-1).$$ 
From  Theorem \ref{thm:eigenestimatedeltabar}, we know that $\lambda+nq\geq 0$. This shows that $A$ and $D$ are nonnegative, as well as $B+C$. 
On the other hand, the system can be equivalently  written as 
$$\Delta\begin{pmatrix} \alpha_m\\\eta\wedge\overline\partial^*\alpha_m\\\eta\wedge\partial^*\alpha_m\\\overline\partial\partial\alpha_{m-1}\end{pmatrix}=\begin{pmatrix}A&2i&-2i&0\\0&B&0&2m\\-4imq &0&C&-2m\\-4i(\lambda+2nq)&\frac{2(\lambda+2mq)}{m}&-\frac{2(\lambda+2nq)}{m}&D\end{pmatrix}\begin{pmatrix} \alpha_m\\\eta\wedge\overline\partial^*\alpha_m\\\eta\wedge\partial^*\alpha_m\\\overline\partial\partial\alpha_{m-1}\end{pmatrix}.$$
When the forms do not vanish, the eigenvalues of the above matrix are given by 
$$\Theta_0(q,\lambda)=2\lambda-4m^2+4mn+2nq+q^2,$$
which is of multiplicity $2$, and two others  
$$\Theta_\pm(q,\lambda)=2\lambda-4m^2+4mn+4m+2nq-2n+q^2\pm 2\sqrt{2\lambda+(n+q)^2}.$$
Notice that the set $\{\alpha_m,\eta\wedge\overline\partial^\ast\alpha_m, \eta\wedge\partial^\ast\alpha_m,\overline\partial\partial\alpha_{m-1}\} $ is orthogonal for the real $L^2$-product, as a result of 
\begin{equation*}
\langle\langle\alpha_m,\overline\partial\partial\alpha_{m-1}\rangle\rangle=\langle\langle\overline\partial^\ast\alpha_m,\partial\alpha_{m-1}\rangle\rangle\stackrel{\eqref{eq:partialstarpartialalphe}}{=}im||\partial\alpha_{m-1}||^2,
\end{equation*}
which is pure imaginary and
$$\langle\langle\eta\wedge\overline\partial^\ast\alpha_m, \eta\wedge\partial^\ast\alpha_m\rangle\rangle=\langle\langle\overline\partial^\ast\alpha_m,\partial^\ast\alpha_m\rangle\rangle=\langle\langle\partial\overline\partial^\ast\alpha_m,\alpha_m\rangle\rangle\stackrel{\eqref{eq:partialstarpartialalphe}}{=}0.$$
Also, notice that the eigenvalues $\Theta_0(q,\lambda), \Theta_\pm(q,\lambda)$ are nonnegative for all $q$ and $\lambda$. Indeed, we can write $\Theta_0(q,\lambda)=\frac{B+C}{2}\geq 0$ and $\Theta_+(q,\lambda)=D+2n+2\sqrt{2\lambda+(n+q)^2}>0$. For $\Theta_-(q,\lambda)=D+2n-2\sqrt{2\lambda+(n+q)^2}$, we can show by a straightforward computation that 
$$(D+2n)^2-4(2\lambda+2(n+q)^2)=\frac{B+C}{2}(D+4m-4)\geq 0,$$
as a result of $\frac{B+C}{2}=D+4n-4m$, which can be easily proven. 

We now discuss the cases when one of these forms vanishes. 
If $\overline\partial^\ast\alpha_m=0$, then the second equation in the system \eqref{eq:system} gives $\overline\partial\partial\alpha_{m-1}=0$, and the fourth equation reduces to $-\frac{2(\lambda+2qn)}{m}\eta\wedge\partial^\ast\alpha_{m}-4i(\lambda+2qn)\alpha_m=0$. 
As $\alpha_m\neq 0$, we deduce that $\lambda+2qn=0$, and the system becomes 
$$\Delta\begin{pmatrix} \alpha_m\\\eta\wedge\partial^*\alpha_m\end{pmatrix}=\begin{pmatrix}A&-2i\\-4imq&C\end{pmatrix} \begin{pmatrix} \alpha_m\\\eta\wedge\partial^*\alpha_m\end{pmatrix}.$$
In this case, if moreover $\partial^\ast\alpha_m=0$, then $q=\lambda=0$ and the number $A=4m(n-m+1)=\Theta_+(0,0)$, is the only eigenvalue of the system. Otherwise, we get two eigenvalues of the matrix that, after replacing $\lambda+2qn=0$, are given  by
$$-2qn+q^2+4mn-4m^2\quad\text{and}\quad -2qn+q^2-2q+4mn-4m^2+4m,$$
which are respectively $\Theta_0(q,-2qn)$ and $\Theta_+(q,-2qn)$ if $n>q$ (or $\Theta_-(q,-2n)$ if $n<q$). Hence, we deduce that the eigenvalues are $\Theta_+(0,0), \Theta_0(q,-2qn)$ and $\Theta_\pm(q,-2qn)$. 
The second case is when $\partial^\ast\alpha_m=0$. 
The third equation in the system then gives $\overline\partial\partial\alpha_{m-1}=-2iq\alpha_m$. 
Replacing this last identity in the fourth equation of the system  gives after combining with the first equation the following identity $2i(\lambda+4qn)\alpha_m=\frac{\lambda}{m}\eta\wedge\overline\partial^\ast\alpha_m$. 
Taking now the Hermitian product with $\alpha_m$ gives that $\lambda+4qn=0$ and $\lambda\overline\partial^\ast\alpha_m=0$. 
In this case, $\lambda=q=0$ and the only eigenvalues are $B=4m(n-m)$ (if $\overline\partial^\ast\alpha_m\neq 0$) or $A=4m(n-m+1)$ (if $\overline\partial^\ast\alpha_m= 0$). 
We deduce in this case that the eigenvalues are $\Theta_0(0,0)$ and $\Theta_+(0,0)$. 
We are now left with the case when $\overline\partial\partial\alpha_{m-1}=0$. The fourth equation of the system  \eqref{eq:system} yields 
$$\frac{2(\lambda+2mq)}{m}\eta \wedge\overline\partial^\ast\alpha_{m}-\frac{2(\lambda+2qn)}{m}\eta\wedge\partial^\ast\alpha_{m}-4i(\lambda+2qn)\alpha_m=0.$$
Taking the Hermitian product with $\alpha_m$ yields $\lambda+2qn=0$ and $(\lambda+2mq)\overline\partial^\ast\alpha_m=0.$ Hence either $\overline\partial^\ast\alpha_m=0$ and the eigenvalues are $\Theta_0(q,-2qn)$ and $\Theta_+(q,-2qn)$ as previously or $\lambda+2mq=0$, which means that $q=\lambda=0$ or $m=n$. In the first case, the eigenvalues are $B$ (if $\overline\partial^\ast\alpha_m\neq 0$), $C$ (if $\partial^\ast\alpha_m\neq 0$) or $A$ (if $\overline\partial^\ast\alpha_m=0$ and $\partial^\ast\alpha_m=0$). That means the eigenvalues are $\Theta_0(0,0)$ or $\Theta_+(0,0)$. When $m=n$, the eigenvalues are $B=-2nq+q^2+2q$, $-2nq+q^2$, or $-2nq+q^2-2q+4n$, which are $\Theta_-(q,-2nq), \Theta_0(q,-2qn)$ and $\Theta_+(q,-2nq)$.  This finishes the proof. 
\end{proof}

\begin{remark}
\begin{enumerate}
\item It is not difficult to check that if $\Delta_{\overline\partial} f=\lambda f$ with $\xi(f)=iq f$, then $\Delta f=(2\lambda+2nq+q^2)f$. Therefore, we can replace in the above expressions of $\Theta_0(q,\lambda),\Theta_\pm(q,\lambda)$ the number $2\lambda+2nq+q^2$ by an eigenvalue of the scalar Laplacian of $M$. Then, they become  independent of $q$. That means, we can express eigenvalues of the Hodge Laplacian in terms of the ones of the scalar Laplacian independently of $q$. 
\item It is easy to see that the eigenvectors of the transposed matrix above yield the following eigenforms of the Laplacian 
\begin{enumerate}
\item The eigenforms of $\Delta$ corresponding to the eigenvalue $\Theta_0(q,\lambda)$ are linear combinations of the two forms 
\begin{gather*}
iq\alpha_m + \eta\wedge\overline\partial^\ast\alpha_m
+\eta\wedge\partial^\ast\alpha_m\text{ and}\\
i(\lambda+2nq)\alpha_m +2(n-m)\eta\wedge\overline\partial^\ast\alpha_m
+m\overline\partial\partial\alpha_{m-1}.
\end{gather*}

\item The eigenspace of $\Delta$ corresponding to eigenvalue $\Theta_\pm(q,\lambda)$ 
is the span of 
\[
2im(q-n\mp l)\alpha_m 
+(q+n\pm l)\eta\wedge\overline\partial^\ast\alpha_m
+(q-n\mp l)\eta\wedge\partial^\ast\alpha_m
+2m\overline\partial\partial\alpha_{m-1}.
\]
where $l:=\sqrt{2\lambda+(n+q)^2}.$
\end{enumerate}
\end{enumerate}
\end{remark}

\color{black}

\section{Example on the round sphere}\label{exampleroundsphere} 
In this section, we consider the example of the round sphere $\mathbb{S}^3$  equipped with its standard Sasakian structure and show that the $\overline\partial$-cohomologies do not vanish. 
We use the same notation as in \cite{Peyerimhoff1993} and \cite{EgidiGittinsHabibPeyerimhoff2023EigEstMagnSchroLapl}.

Let $\mathbb{S}^3\subset%
\mathbb{C}^2$ be the sphere equipped with the metric $g$ of curvature $1$. The Reeb vector field that defines the Sasakian structure is given by  $%
\xi=\sum_{j=1}^2(-y_j\partial_{x_j}+x_j\partial_{y_j})$ at each point $(z_1,z_2) = (x_1 + y_1 i, x_2 + y_2 i) \in \mathbb{S}^3$.  The three vector fields
\begin{eqnarray*}
\xi=Y_2 &=& -y_1 \partial_{x_1} + x_1 \partial_{y_1} - y_2 \partial_{x_2} + x_2
\partial_{y_2}, \\
Y_3 &=& -y_2 \partial_{x_1} - x_2 \partial_{y_1} + y_1 \partial_{x_2} + x_1
\partial_{y_2}, \\
Y_4 &=& x_2 \partial_{x_1} - y_2 \partial_{y_1} - x_1 \partial_{x_2} + y_1
\partial_{y_2},
\end{eqnarray*}
define a direct orthonormal basis of $T_{(z_1,z_2)}\mathbb{S}^3$. A direct computation of the Christoffel symbols of the Levi-Civita connection gives that
\begin{equation*}  
\nabla_{Y_j} Y_k= \sigma_{jk} Y_l
\end{equation*}
with $\{j,k,l\} = \{2,3,4\}$ for $k \neq j$, $\sigma_{jj} = 0$ and $%
\sigma_{23}=-\sigma_{24} =-1$, $\sigma_{32}=\sigma_{43} = - \sigma_{34} =
-\sigma_{42} = 1$. 
In particular, we deduce that 
\begin{equation*}
d\eta=2 Y_3^\flat \wedge Y_4^\flat,\,\, dY_3^\flat=-2 \eta\wedge Y_4^\flat,\,\,
dY_4^\flat=2\eta\wedge Y_3^\flat\quad\text{and}\quad \delta Y_j=0,
\end{equation*}
for $j\in\{2,3,4\}$. Hence, we get that $\mathcal{L}_\xi
Y_3^\flat=-2Y_4^\flat$ and $\mathcal{L}_\xi Y_4^\flat=2Y_3^\flat$. This
system can be equivalently written as 
\begin{equation}  \label{eq:exteriory2}
d\eta=4i (Y_3^-)^\flat \wedge (Y_3^+)^\flat,\,\, d(Y_3^\pm)^\flat=\mp 2i \eta\wedge (Y_3^\pm)^\flat,\, \quad%
\text{and}\quad \delta (Y_3^\pm)^\flat=0.
\end{equation}
We also have that $\mathcal{L}_\xi (Y_3^+)^\flat=-2i(Y_3^+)^\flat$ and $\mathcal{L}_\xi
(Y_3^-)^\flat=2i(Y_3^-)^\flat$.  Equations \eqref{eq:exteriory2} give that 
\begin{equation*}
\partial (Y_3^+)^\flat=0,\,\, \overline{\partial}(Y_3^+)^\flat=0,\,\, d^{1,1,-1}(Y_3^+)^\flat=0,\,\,
d^{0,0,1}(Y_3^+)^\flat=-2i\eta\wedge (Y_3^+)^\flat,\,\, \overline\partial^\ast (Y_3^+)^\flat=0. 
\end{equation*}
This shows that $(Y_3^+)^\flat\in \mathcal{H}_{\overline\partial}^{0,1,0}(-2)$ and $(Y_3^-)^\flat\in \mathcal{H}_{\overline\partial}^{1,0,0}(2)$. 

Recall that the eigenvalues of the scalar Laplacian are
given by $\ell (\ell +2)$ for $\ell \in\mathbb{N}\cup \{0\}$ with multiplicity $(\ell +1)^2$. The
eigenfunctions are then defined by $\phi_{\ell ,p}:= u^p v^{\ell -p}$, for $p \in
\{0,\ldots, \ell \}$, where $u(z_1, z_2):= az_1+bz_2$ and $v(z_1, z_2):=b \bar z_1 -a\bar z_2$ for some $(a,b)\neq (0,0)$. For all other choices of $\ell , p$, we set $\phi_{\ell ,p} \equiv 0$.
A straightforward computation yields 
\begin{eqnarray}  \label{eq:Y2phi}
Y_2(\phi_{\ell ,p}) &=& i(2p-\ell )\phi_{\ell ,p},  \notag \\
Y_3(\phi_{\ell ,p}) &=& ip\phi_{\ell ,p-1}+i(\ell -p)\phi_{\ell ,p+1},  \notag \\
Y_4(\phi_{\ell ,p}) &=& -p\phi_{\ell ,p-1}+(\ell -p)\phi_{\ell ,p+1}.  \notag
\end{eqnarray}
Therefore, we deduce that 
\begin{eqnarray*}
d\phi_{\ell ,p}&=&i(2p-\ell )\phi_{\ell ,p}\eta+i(p\phi_{\ell ,p-1}+(\ell -p)%
\phi_{\ell ,p+1})(Y_3)^\flat+((\ell -p)\phi_{\ell ,p+1}-p\phi_{\ell ,p-1})(Y_4)^\flat \\
&=&i(2p-\ell )\phi_{\ell ,p}\eta+2ip\phi_{\ell ,p-1}(Y_3^-)^\flat+2i(\ell -p)\phi_{\ell ,p+1}(Y_3^+)^\flat,
\end{eqnarray*}
and, thus, 
\begin{equation*}  
\partial\phi_{\ell ,p}=2ip\phi_{\ell ,p-1}(Y_3^-)^\flat,\,\, \overline{\partial}%
\phi_{\ell ,p}=2i(\ell -p)\phi_{\ell ,p+1}(Y_3^+)^\flat,\,\, d^{1,1,-1}\phi_{\ell ,p}=0,\,\,
d^{0,0,1}\phi_{\ell ,p}=i(2p-\ell )\phi_{\ell ,p}\eta.
\end{equation*}
Hence, we get that $\overline\partial\phi_{\ell ,\ell }=0$, that is, $\phi_{\ell ,\ell }\in \mathcal{H}^{0,0,0}_{\overline\partial}(\ell )$ for $\ell \in\mathbb{N}\cup \{0\}$. As a consequence, we find that $\overline\partial (\phi_{\ell ,\ell }\eta)=0$. From the equality  $$\delta(\phi_{\ell ,\ell }\eta)=\phi_{\ell ,\ell }\delta\eta-\xi(\phi_{\ell ,\ell })=\ell \phi_{\ell ,\ell },$$
we get  $\overline\partial^\ast(\phi_{\ell ,\ell }\eta)=0$ with $\mathcal{L}_\xi(\phi_{\ell ,\ell }\eta)=i\ell \phi_{\ell ,\ell }\eta$. Thus, $\phi_{\ell ,\ell }\eta\in \mathcal{H}_{\overline\partial}^{0,0,1}(\ell ).$
For differential forms of type $(0,1,0)$, we have  $\overline\partial(\phi_{\ell ,p}Y_3^+)=0$ and
$$\delta(\phi_{\ell ,p}(Y_3^+)^\flat)=-Y_3^+(\phi_{\ell ,p})=-ip\phi_{\ell ,p-1}.$$
Hence, $\overline\partial^\ast(\phi_{\ell ,0}(Y_3^+)^\flat)=0$ with $\mathcal{L}_\xi(\phi_{\ell ,0}(Y_3^+)^\flat)=-i(\ell +2)\phi_{\ell ,0}(Y_3^+)^\flat$. Thus, $\phi_{\ell ,0}(Y_3^+)^\flat\in \mathcal{H}_{\overline\partial}^{0,1,0}(-(\ell +2)).$
By taking the conjuguate, we have  $\phi_{\ell ,\ell }(Y_3^-)^\flat\in \mathcal{H}_{\overline\partial}^{1,0,0}(\ell +2).$ 

 It is not difficult to check that $\Delta_{\overline\partial}\phi_{l,p}=2(l-p)(p+1)\phi_{l,p}$. According to Theorem \ref{eigenvalue}, the eigenvalues $\Theta_0=l(l+2), \Theta_-=l^2$ and $\Theta_+=(l+2)^2$ belong to the spectrum. However, it is known that the spectrum of the Hodge operator on the $3$-sphere on $2$-forms (or $1$-forms) is given by the families $\{l(l+2)|\, l\geq 1\}$ and  $\{(l+2)^2|\,\, l\geq 0\}$ \cite{GM75}. 

\section{Appendix:  Dolbeault operators over the cone}\label{appendix}
In this part, we will interpret the formulas for $\partial _{C}$ and $\overline{\partial _{C}}$ over the 
cone $C\left( M\right) =\left( 0,\infty \right) \times M$
of a Sasakian manifold $(M,g,\xi)$. Since it is a K\"ahler manifold, we will relate these formulas to the usual exterior differentials $d^{1,0}$ and $d^{0,1}$. As noted in \cite[Section 6.4]{OrneaVerbitsky2024PrincLocConfKahlGeom} and
in \cite[Section 6.5.3]{BoyerGalicki2008SasakianGeometry}, one can realize
the Sasakian structure as the restriction of the K\"{a}hler structure. The metric
on $C\left( M\right) $ is $g^{C\left( M\right) }=dt^{2}+t^{2}g$, where $t$
is the coordinate on $\left( 0,\infty \right) $. The complex structure on $%
C\left( M\right) $ is defined to be $J:TC\left( M\right) \rightarrow
TC\left( M\right) $ by 
$$J\left( \partial _{t}\right) =\frac{1}{t}j_{\ast }\xi,\,\,\,
J\left( \frac{1}{t}\xi \right) =-\partial _{t},\,\,\, J\left( j_{\ast
}X\right) =j_{\ast }\phi \left( X\right) $$ for 
$X\in \xi ^{\bot }\subseteq
TM $, where $j:M\hookrightarrow C\left( M\right) $ is the inclusion. The K%
\"{a}hler form $\omega ^{C\left( M\right) }$ is defined by $\omega ^{C\left(
M\right) }\left( Y_{1},Y_{2}\right) =g^{C\left( M\right) }\left(
JY_{1},Y_{2}\right) $. We identify forms on $M$ with their pullbacks to $%
C\left( M\right) $. The form $\eta $ on $M$ is the pullback $j^{\ast
}\left( v\lrcorner \omega ^{C\left( M\right) }\right) $, where $v=t\partial
_{t}$ is the (holomorphic) radial vector field on $C\left( M\right) $,
acting by radial homotheties. 

Let $%
p:C\left( M\right) \twoheadrightarrow M$ be the projection $\left(
t,x\right) \mapsto x$. We define the complex frame on  $TC\left( M\right) \otimes \mathbb{C}$ by
\begin{eqnarray*}
Z_{0} &=&\frac{1}{\sqrt{2}}\left(
\partial _{t}-\frac{i}{t}j_{\ast }\xi \right)
\end{eqnarray*}
with its dual 
\begin{eqnarray*}
Z^{0} &=&Z_{0}^{\ast }=\frac{1}{\sqrt{2}}\left( dt+\frac{i}{t}p^{\ast }\eta
\right).
\end{eqnarray*}
We choose a local transverse orthonormal frame $\{e_{1},\phi
e_{1},...,e_{n},\phi e_{n}\}$ on $\xi^\perp$ and let,  for each $k=1,\ldots, n$ the complex frame  $Z_{k}=\frac{1}{\sqrt{2}}\left( e_{k}-i\phi e_{k}\right)$.  Then $\left( Z_{0},\frac{1}{t}j_{\ast }Z_{1},...,\frac{%
1}{t}j_{\ast }Z_{n}\right) $ gives an orthonormal K\"{a}hler frame for $%
T^{1,0}C\left( M\right) $ with associated coframe $\left( Z^{0},
\frac{1}{t}p^{\ast }Z^{1},...,\frac{1}{t}p^{\ast }Z^{n}\right) $.

Now, we consider on $M$ the composition $j^{\ast }\circ d^{1,0}\circ p^{\ast }$ and will prove that this composition is exactly $\partial_C$ defined previously on all differential forms. It is not difficult to check that $d\eta=d^{1,1,-1}\eta=-2i\sum
Z_{j}^{\ast }\wedge \overline{Z_{j}^{\ast }}$. Thus, 
\begin{eqnarray*}
dp^{\ast }\eta &=&p^{\ast }d\eta \\
&=&-2i\sum p^{\ast }Z_{j}^{\ast }\wedge \overline{p^{\ast }Z_{j}^{\ast }}\in
T^{1,1}C\left( M\right).
\end{eqnarray*}
Therefore, we get that 
$$dZ^0=\frac{i}{\sqrt{2}t}d(p^*\eta)=\frac{\sqrt{2}}{t}\sum p^{\ast }Z_{j}^{\ast }\wedge \overline{p^{\ast }Z_{j}^{\ast }}\in
T^{1,1}C\left( M\right).$$ 
Hence, as $Z^0$ is of type $(1,0)$, we deduce that 
$$d^{1,0}Z^0=0\quad\text{and}\quad d^{0,1}Z^0=\frac{\sqrt{2}}{t}\sum p^{\ast }Z_{j}^{\ast }\wedge \overline{p^{\ast }Z_{j}^{\ast }}.$$
Similarly, we show that%
\begin{equation*}
d^{0,1}\overline{Z^{0}}=0 \quad\text{and}\quad d^{1,0}\overline{Z^{0}}=-\frac{\sqrt{2}}{t}\sum p^{\ast }Z_{j}^{\ast }\wedge 
\overline{p^{\ast }Z_{j}^{\ast }}.
\end{equation*}%
Now, by writing $p^*\eta=\frac{t}{i\sqrt{2}}(Z^0-\overline{Z^{0}})$, we compute 
\begin{equation}\label{eq:d10eta}
 d^{1,0}p^*\eta=\frac{t}{i\sqrt{2}}(d^{1,0}Z^0-d^{1,0}\overline{Z^{0}})=\frac{1}{i} \sum p^{\ast }Z_{j}^{\ast }\wedge 
\overline{p^{\ast }Z_{j}^{\ast }}.  
\end{equation}
Now, take any differential $k$-form $\alpha+\eta\wedge\beta$ on $M$ with $\xi\lrcorner\alpha=0$ and $\xi\lrcorner\beta=0$. We compute
\begin{eqnarray*}
j^*d^{1,0}p^*(\alpha+\eta\wedge\beta)&= &j^*d^{1,0}p^*\alpha+j^*d^{1,0}(p^*\eta\wedge p^*\beta) \\
&=& d^{1,0,0}\alpha+j^*d^{1,0}p^*\eta\wedge  j^*p^*\beta-j^*p^*\eta\wedge j^*d^{1,0}p^*\beta \\
&\stackrel{\eqref{eq:d10eta}}{=}&d^{1,0,0}\alpha+\frac{1}{i}\sum Z_{j}^{\ast }\wedge 
\overline{Z_{j}^{\ast }}\wedge \beta-\eta\wedge d^{1,0,0}\beta\\
&=&d^{1,0,0}\alpha+\frac{1}{2}d\eta\wedge \beta-\eta\wedge d^{1,0,0}\beta\\
&=&(d^{1,0,0}+\frac{1}{2}d^{1,1,-1})(\alpha+\eta\wedge\beta).
\end{eqnarray*}
In the last computations, we use the fact that $d^{1,1,-1}(\alpha+\eta\wedge\beta)=d\eta\wedge(\xi\lrcorner(\alpha+\eta\wedge\beta))=d\eta\wedge\beta$. Therefore, $\partial _{C}=d^{1,0,0}+%
\frac{1}{2}d^{1,1,-1}=j^{\ast }d^{1,0}p^{\ast }$ and, similarly, $\overline{\partial _{C}}=d^{0,1,0}+\frac{1}{2}d^{1,1,-1}=j^{\ast }d^{0,1}p^{\ast }$
for all differential forms.


\bibliographystyle{amsplain}
\bibliography{habibRichardsonWolak}

\end{document}